\documentclass[10pt,reqno,final]{amsart}
\usepackage{epsfig,amssymb,amsmath,version}
\usepackage{amssymb,version,graphicx,fancybox,mathrsfs}
\usepackage[notcite,notref]{showkeys}
\usepackage{url,hyperref}
\usepackage{subfigure}
\usepackage{color}
\usepackage{stmaryrd}
\usepackage{multirow}
\usepackage{booktabs,siunitx}
\usepackage{bm} 

\usepackage{booktabs}

\usepackage[ruled,vlined]{algorithm2e}

\catcode`\@=11 \theoremstyle{plain}
\@addtoreset{equation}{section}   

\@addtoreset{figure}{section}
\renewcommand\thefigure{\thesection.\@arabic\c@figure}
\newtheorem{thm}{\bf Theorem}

\newtheorem{cor}{\bf Corollary}
\newtheorem{prop}{Proposition}[section]

\newtheorem{lmm}{\bf Lemma}

\theoremstyle{remark}
\newtheorem{rem}{\bf Remark}[section]

\def \epsilon {{\varepsilon}}

\definecolor{bgblue}{rgb}{0.04,0.39,0.54}
\definecolor{lired}{rgb}{0.3, 0.0, 0.0}
\definecolor{ligreen}{rgb}{0.0, 0.3, 0.0}
\definecolor{liblue}{rgb}{0.9, 1.0, 1.0}
\definecolor{gray}{rgb}{0.6, 0.6, 0.6}
\definecolor{sky}{rgb}{0.3, 1.0, 1.0}
\definecolor{bunhong}{rgb}{1.0, 0.3, 1.0}
\definecolor{yellow}{rgb}{0.97, 1, 0.0}
\definecolor{liyellow}{rgb}{0.9, 0.8, 0.0}
\definecolor{cengse}{rgb}{0.00,0.40,0.29}

\newcommand{\bs}[1]{\boldsymbol{#1}}

\renewcommand \wedge \times
\newcommand{\tensor}[1]{\overline {\overline{#1}}}

\graphicspath{{../figs/} {../figs/case1/} {../figs/nsoli_convergence/}}

\begin{document}
\bibliographystyle{plain}

{\title[Rotational discrete gradient method for Oseen-Frank gradient flows] {A second-order length-preserving and unconditionally energy stable rotational discrete gradient method for Oseen-Frank gradient flows}
\author[
	J. Xu,\,    X. Yang\,  $\&$\,  Z. Yang
	]{
		\;\; Jie Xu${}^{1}$,   \;\;  Xiaotian Yang${}^{2}$ \;\; and\;\; Zhiguo Yang${}^{2,*}$
		}
	\thanks{${}^{*}$ Corresponding author.
	\\
	\indent ${}^{1}$ LSEC and NCMIS, Institute of Computational Mathematics and Scientific/Engineering Computing (ICMSEC), Academy of Mathematics and Systems Science (AMSS), Chinese Academy of Sciences, Beijing, China. Emails: xujie@lsec.cc.ac.cn (J. Xu).
		 \\
		\indent${}^{2}$ School of Mathematical Sciences, MOE-LSC and CMA-Shanghai, Shanghai Jiao Tong University, Shanghai 200240, China. Email: yang$\_$xiaotian@sjtu.edu.cn (X. Yang) and yangzhiguo@sjtu.edu.edu (Z. Yang)
		}
		}


\begin{abstract}
We present a second-order strictly length-preserving and unconditionally energy-stable rotational discrete gradient (Rdg) scheme for the numerical approximation of the Oseen-Frank gradient flows with anisotropic elastic energy functional.  Two essential ingredients of the Rdg method are reformulation of the length constrained gradient flow into an unconstrained rotational form and discrete gradient discretization for the energy variation. Besides the well-known mean-value and Gonzalez discrete gradients, we propose a novel Oseen-Frank discrete gradient, specifically designed for the solution of Oseen-Frank gradient flow. We prove that the proposed Oseen-Frank discrete gradient satisfies the energy difference relation, thus the resultant Rdg scheme is energy stable. Numerical experiments demonstrate the efficiency and accuracy of the proposed Rdg method and its capability for providing reliable simulation results with highly disparate elastic coefficients.

\end{abstract}
 \maketitle

\section{Introduction}
The dynamics of liquid crystals involve the evolution of local anisotropy generated by nonuniform orientational distribution.
For uniaxial nematics, the equilibrium orientational distribution is axisymmetric and is allowed to rotate freely as a whole. 
A simplified setting for the dynamics of uniaxial nematics is to assume that the local anisotropy is kept at the equilibrium state and only rotations of the state is allowed.
Under this rationale, the local anisotropy is sufficiently described by a unit vector field $\bs{n}(\bs{x})$.
The well-known Ericksen-Leslie model \cite{ericksen1961conservation,leslie1968some} couples the velocity and the unit vector field.
The evolution of the vector field is given by 
\begin{subequations}\label{eq: elmodel}
\begin{align}
&\bs{n}\times \Big(   \gamma_1\big( \bs n_t +\bs v\cdot \nabla \bs n+\tensor{\Omega}\cdot \bs n \big)+\gamma_2 \tensor{\tau} \cdot \bs n +\frac{\delta \mathcal{F}[\bs n]}{\delta \bs{n}} \Big)=0, \label{eq: directeq}\\
& |\bs n|=1. \label{eq: unit}
\end{align}
\end{subequations}
In the above,
$\tensor{\Omega}= (\nabla \bs v-\nabla \bs v^{\intercal})/2$, and $\tensor{\tau}= (\nabla \bs v+\nabla \bs v^{\intercal})/2 $, where $\bs{v}$ is the velocity of the fluid governed by the Navier--Stokes equation that we do not write down here. 
The force by the interaction of local anisotropy is characterized by variational derivative of the Oseen--Frank energy \cite{Oseen1933,Frank1958}, 
\begin{equation}\label{eq: osfrank}
\mathcal{F}[\bs n]=\frac{1}{2}\int_{\Omega}k_1(\nabla \cdot \bs{n})^{2}+k_2|\bs{n}\cdot (\nabla \times \bs{n})|^{2}+k_3|\bs{n}\times (\nabla \times \bs{n})|^{2}+(k_2+k_4)\big[{\rm tr}\big((\nabla \bs{n})^{2}\big)-(\nabla \cdot \bs{n})^{2} \big]dV.
\end{equation}
The first three terms can be explained as excess energy density for three typical deformations specifically for unit vector fields: splay, twist and bend \cite{deGennesProst1993}.
Therefore, it is the three constants $k_1,k_2,k_3$ that characterize the elasticity of a certain material. 
They are closely related to physical parameters, whose relative magnitudes may vary in a wide range by previous experimental or theoretical results \cite{Frederiks,Frederiks_prl1969,JPCB2010,kapanowski1997statistical,ECExp_pre2012,ECExp_pre2011,SoftM2013,ECExp_pre2010_2,han2015from,xu2018tensor,xu2018calculating,xu2018onsager,Li2023frame}. 
The last term can be rewritten as a surface integral, which is usually not considered as it vanishes under periodic or some other commonly adopted boundary conditions.

When the velocity is small, one may approximate by assuming $\bs{v}=0$, so that the Ericksen-Leslie model is reduced to
\begin{subequations}\label{eq: gfmodel}
\begin{align}
&\bs{n}\times \bs n_t=- \bs n \times \frac{\delta \mathcal{F}[\bs n]}{\delta \bs{n}}, \label{eq: gflow}\\
& |\bs n|=1, \label{eq: unit2}
\end{align}
\end{subequations}
where $\gamma_1$ is also eliminated with suitable rescaling.
This equation can be interpreted as a gradient flow driven by the Oseen--Frank energy with an explicit length constraint, which we shall clarify later.
Notably, it is desirable for a numerical scheme to keep the vector length and the energy dissipation.
In particular, the deviation of vector length would bring ambiguity when explaining the results because the pictures of splay, twist and bend terms are no longer valid for vector fields with varying lengths. 
However, both the $k_i$ terms and the length constraint give rise to strong nonlinearity that is not easy to handle. 

Most of the previous works on the analysis and numerical methods of nematic liquid crystal problems study the isotropic elasticity (see e.g. \cite{lin_liu1995, alouges1997new, Liu_and_Walkington-2000, Du_Guo_and_Shen-2001, Feng_and_Prohl-2004, Guill2013,Zhao_Yang_Li_and_Wang-2016}), i.e. $k_1=k_2=k_3=k$ and $k_4=0$ making $\mathcal{F}[\bs n]$ reduce to the squared gradient, 
\begin{equation}\label{eq: DirEnergy}
\mathcal{F}_D[\bs n]=\frac{k}{2}\int_{\Omega}|\nabla \bs{n}|^{2}dV.
\end{equation}
There are only limited studies on numerical methods for anisotropic elasticity (see \cite{ramage2013preconditioned, adler_et_al-2015}).
Even assuming the isotropic elasticity, to the best of the authors' knowledge, in existing works the length contraint and energy dissipation are rarely satisfied simultaneously. 
%
%
%
One popular choice is to relax the model by substituting the explicit length contraint with a penalty term in the energy \cite{Liu_and_Walkington-2000, Du_Guo_and_Shen-2001, Feng_and_Prohl-2004, Guill2013,Zhao_Yang_Li_and_Wang-2016}. 
As a result, we obtain an equation of the Ginzburg--Landau type, 
\begin{equation}\label{eq: GinzLaudau}
\bs n_t=-k \Delta \bs n+F'_{\epsilon}(\bs n),\quad F_{\epsilon}(\bs n)=\big(|\bs n|^2-1 \big)^2/\big(4\epsilon^2 \big).
\end{equation}
We can see that the nonlinearity on length constraint is converted into a nonlinear term in the energy, and the spatial derivative term is linear when the isotropic elasticity is assumed. 
The problem then becomes standard as one only needs to focus on energy dissipation, for which various techniques are available, including convex splitting \cite{Elliott_and_Stuart-1993,Eyre-1998}, linear stablization \cite{Shen_Yang-2010, Cai_Shen_and_Xu-2017}, and auxiliary variable such as IEQ \cite{Yang-2016, Zhao_Yang_Li_and_Wang-2016} and SAV \cite{Shen_Xu_Yang-2018,Shen_Xu_Yang-2019}. 
However, it is worthwhile to note that the relaxation allows the vector length to deviate. 
To date, it is still not definite how such a relaxation approximate the original Oseen--Frank gradient flow.
On the other hand, some works put more emphasis on the unit-vector constraint.
One simple idea is projection \cite{E_and_Wang-2001}, which follows the predictor-corrector routine. Ignoring the unit-vector constraint, one firstly obtains in the prediction step an intermediate approximate solution, followed by the normalization on the vector length.
Despite its simplicity, it is highly nontrivial to extend the projection method to high-order schemes.
Another straightforward approach is to adopt the spherical coordinates method by representing the vector field in terms of the orientation angles
$ \bs n=\cos (\theta) \cos (\phi) \bs e_x +\cos (\theta) \sin (\phi) \bs e_y +\sin(\theta) \bs e_z $, and to numerically evolve the angles. 
The unit-vector constraint can then be naturally guaranteed. 
However, as pointed out in \cite{ramage2013preconditioned}, such a representation degenerates when $\cos (\theta)=0$ and the orientation angle is discontinuous across $\theta+\pi$ and $\phi+2\pi$.
In addition, the energy dissipation cannot be guaranteed theoretically mainly because the energy takes a form much more complicated. 
It is also possible to introduce a Lagrange multiplier for the length constraint, which transforms the original equation into a saddle-point form \cite{Badia2011a, badia2011overview}. 
Currently, this approach requires solving within the framework of mixed elements, which is designed specially for the isotropic squared gradient term. 
To investigate and compare the dynamics of nematics formed by different materials, it calls for reliable and efficient schemes to deal with anisotropic elasticity, but it turns out that the existing approaches are far from satisfactory.

In this paper, we propose a novel second-order rotational discrete gradient (Rdg) scheme for the Oseen-Frank gradient flows with anisotropic elasticity, which is strictly length-preserving meanwhile unconditionally energy stable.
This is achieved by reformulating the gradient flow with unit-vector constraint into an unconstrained rotational form and dealing with the variational derivative via discrete gradients.
For functional of multiple variables, the construction of discrete gradient is not unique. 
We compare three different discrete gradients in the method we propose, including the widely-used mean-value and Gonzalez discrete gradients, together with one particularly designed for the Oseen-Frank energy.
Although theoretically they all satisfy the two desired properties, in practice nonlinear equations need to be solved, leading to difference in performance. 
Our numerical tests indicate that the specially designed Oseen-Frank discrete gradient outperforms the other two discrete gradients in length error and efficiency. 
On the other hand, we examine the dynamics with different elastic constants, and find that anisotropic elasticity can lead to completely different dynamics. 
Since the Oseen--Frank gradient flow characterizes the dissipation of the unit vector field in the Ericksen--Leslie model, the proposed method is expected to be a cornerstone for efficient and robust schemes of the Ericksen--Leslie equation. 

The rest of the paper is organized as follows. 
In Section \ref{sect: rdg}, we introduce an equivalent unconstrained rotational form of general gradient flows with unit-vector constraint. We then present the Rdg scheme for temporal discretization of the reformulated system and rigorously prove that it is strictly length preserving and unconditionally energy stable. A specifically designed second-order Oseen-Frank discrete gradient approximation for the energy variation is presented. Efficient time-adaptive strategy is also discussed. In Section \ref{sect: num}, we conduct ample numerical experiments of Oseen-Frank gradient flow involving highly disparate elastic coefficients to demonstrate the accuracy, efficiency, exact length preservation and energy stability of Rdg method. Concluding remarks are given in Section \ref{sect: concl}.


\section{The rotational discrete gradient method}\label{sect: rdg}

In this section, we shall develop an unconditionally energy-stable and length-preserving rotational discrete gradient method (Rdg) for the system \eqref{eq: gfmodel}.
The proposed method is constructed upon the rotational form of the system \eqref{eq: gfmodel} where the length constraint becomes intrinsic.
The rotational form leads to a natural second-order time discretization that maintains the vector length while keeping the structure of the gradient flow.
As a fine result, the techniques for energy dissipative schemes can be built in.
Since the discretization for the rotational form is already nonlinear, it is meaningless to pursue linear discretization of the variational derivative. 
Thus, we choose to construct discrete gradients that are second-order accurate in time to arrive at the rotational discrete gradient method. 
This method can be applied to general length-constrained gradient flows without major difficulty. 
Below, we first write down the rotational form and the corresponding discretization, followed by discussions on discrete gradients. 

\subsection{Rotational form}
We begin with writing down the rotational form and showing the equivalence with system \eqref{eq: gfmodel}. 
\begin{prop}
The system \eqref{eq: gfmodel} is equivalent to the following unconstrained equation 
\begin{equation}\label{rotform}
\bs n_t=- \Big( \bs n\times    \frac{\delta \mathcal{F}[\bs n]}{\delta \bs n} \Big)\times \bs n ,
\end{equation}
provided that the initial vector field is of unit length, i.e. $|\bs n(\bs x,0)|=1.$
\end{prop}
\begin{proof}
On one hand, we start from \eqref{eq: gfmodel} to derive \eqref{rotform}. 
Using the identity of double cross product, we have 
\begin{equation}\label{eq:aba}
(\bs a\times \bs b) \times \bs a=|\bs{a}|^2\bs b-(\bs b\cdot \bs a)\bs a=\bs b \cdot(|\bs{a}|^2I-\bs a\otimes \bs a),\quad \forall \bs a,\bs b\in \mathbb{R}^3.
\end{equation}
Taking the right cross product of $\bs n$ on both sides of \eqref{eq: gflow}, we obtain 
\begin{equation}\label{eq: rotformt1}
\bs n_t-\frac{1}{2} \big(|\bs n|^2 \big)_t\, \bs n=-  \Big(  \bs n\times  \frac{\delta \mathcal{F}[\bs n]}{\delta \bs n} \Big)  \times \bs n   .
\end{equation}
The constraint \eqref{eq: unit2} implies $\big(|\bs n|^2 \big)_t=0$, so that we arrive at \eqref{rotform}.


On the other hand, one can recover \eqref{eq: gfmodel} from \eqref{rotform}.
Taking the left cross product by $\bs n$ on \eqref{rotform}, we deduce that 
\begin{equation}
\bs n \times \bs n_t=-\bs n \times \Big( \bs n\times    \frac{\delta \mathcal{F}[\bs n]}{\delta \bs n} \Big)\times \bs n=-\bigg( \bs n \times  \frac{\delta \mathcal{F}[\bs n]}{\delta \bs n} -\Big((\bs n \times \bs n)\cdot  \frac{\delta \mathcal{F}[\bs n]}{\delta \bs n}\Big) \bs n    \bigg) =- \bs n \times  \frac{\delta \mathcal{F}[\bs n]}{\delta \bs n},
\end{equation}
for which we have utilized \eqref{eq:aba} and the following identities
\begin{equation}\label{eq: eleid}
\bs a \times \bs b \cdot \bs c=\bs c \times \bs a \cdot \bs b,\quad \bs a\times \bs a=0, \quad \forall \bs a,\bs b,\bs c \in \mathbb{R}^3.
\end{equation}
Then, we take the dot product with $\bs n$ on \eqref{rotform}, which yields  
\begin{equation*}
\frac{1}{2} \big(|\bs n|^2 \big)_t=- \Big( \bs n\times    \frac{\delta \mathcal{F}[\bs n]}{\delta \bs n} \Big)\times \bs n \cdot \bs n=0,
\end{equation*}
where identities in \eqref{eq: eleid} have also been used.
Together with the initial condition that $|\bs n(\bs x,0)|=1$, the length constraint is derived, so that we obtain the original constrained system \eqref{eq: gfmodel}.
\end{proof}

Next, let us recognize the gradient flow structure from the rotational form.
The right-hand side of \eqref{rotform} is actually the opposite of variational derivative under the constraint $|\bs{n}|=1$. 
Actually, let us consider a general energy functional $\mathcal{E}[\bs{v}(\bs{x})]$ under the constraint $|\bs{v}|=C$ is constant.
Choose a parameterized curve $\bs{v}(\bs{x},s)$ such that $|\bs{v}(\bs{x},s)|=C$ and $\bs{v}(\bs{x},0)=\bs{v}$. 
The definition of variational derivative without constraint yields
\begin{equation}
  \lim_{s\to 0}\frac{\mathcal{E}[\bs{v}(\bs{x},s)]-\mathcal{E}[\bs{v}(\bs{x},0)]}{s}=\int_{\Omega}\frac{\delta\mathcal{E}[\bs{v}]}{\delta\bs{v}}\cdot\frac{d\bs{v}}{ds}dV. 
\end{equation}
Taking the derivative w.r.t. $s$ on $|\bs{v}(\bs{x},s)|^2=C^2$ yields
\begin{equation}
  \bs{v}\cdot \frac{d\bs{v}}{ds}=0. 
\end{equation}
Therefore, 
\begin{align}
  \int_{\Omega}\frac{\delta\mathcal{E}[\bs{v}]}{\delta\bs{v}}\cdot\frac{d\bs{v}}{ds}dV
  =&\,\int_{\Omega}\frac{\delta\mathcal{E}[\bs{v}]}{\delta\bs{v}}\cdot \Big(I-\frac{\bs{v}}{|\bs{v}|}\otimes \frac{\bs{v}}{|\bs{v}|} \Big)\frac{d\bs{v}}{ds}dV\nonumber\\
  =&\,\int_{\Omega}\Big(I-\frac{\bs{v}}{|\bs{v}|}\otimes \frac{\bs{v}}{|\bs{v}|} \Big)\frac{\delta\mathcal{E}[\bs{v}]}{\delta\bs{v}}\cdot\frac{d\bs{v}}{ds}dV\nonumber\\
  =&\,\int_{\Omega}\Big(\frac{\bs{v}}{|\bs{v}|}\times \frac{\delta\mathcal{E}[\bs{v}]}{\delta\bs{v}} \Big)\times \frac{\bs{v}}{|\bs{v}|}\cdot\frac{d\bs{v}}{ds}dV,
\end{align}
where for the last equality we have used \eqref{eq:aba}. 
Note that $\big(I-\frac{\bs{v}}{|\bs{v}|}\otimes \frac{\bs{v}}{|\bs{v}|} \big)\frac{\delta\mathcal{E}[\bs{v}]}{\delta\bs{v}}=\big(\frac{\bs{v}}{|\bs{v}|}\times \frac{\delta\mathcal{E}[\bs{v}]}{\delta\bs{v}} \big)\times \frac{\bs{v}}{|\bs{v}|}$ lies within the tangent space of the surface $|\bs{v}|=C$, it indeed gives the variational derivative under the constraint. 

The two properties of \eqref{rotform} is summarized in the following theorem. 
\begin{thm}
The gradient flow \eqref{rotform} with initial condition $|\bs n(\bs x,0)|=1$ obeys two fundamental physical properties, i.e.
\begin{itemize}
\item[(1)] the length constraint $|\bs n(\bs x ,t)|=1$ for $t\geq0$;
\item[(2)] the energy dissipation law 
\begin{equation}\label{eq: dislawcont}
\frac{d\mathcal{F}[\bs n]}{dt}=-\int_{\Omega} \Big|   \frac{\delta \mathcal{F}[\bs n]}{\delta \bs n}\times \bs n   \Big|^2 dV\leq 0.
\end{equation}
\end{itemize}
\end{thm}
\begin{proof}
The length constraint has been shown above. 
For the energy dissipation, we take inner product with $ {\delta \mathcal{F}[\bs n]}/{\delta \bs n}$ on \eqref{rotform} to obtain
\begin{equation}
\frac{d\mathcal{F}[\bs n]}{dt}=-\int_{\Omega}\frac{\delta \mathcal{F}[\bs n]}{\delta \bs n}  \cdot \bs n\times \Big(   \frac{\delta \mathcal{F}[\bs n]}{\delta \bs n}\times \bs n   \Big)dV=-\int_{\Omega}    \Big|\frac{\delta \mathcal{F}[\bs n]}{\delta \bs n}\times \bs n \Big|^2   dV\leq 0,
\end{equation}
where the elementary identity \eqref{eq: eleid} has been invoked. 
\end{proof}

\subsection{Length-preserving and energy stable discretization}
We now propose a second-order unconditionally energy stable and length-preserving scheme for the rotational form \eqref{rotform} of the gradient flow. 
Denote by $\tau$ the time step, and use the superscript $m$ to represent the approximation at the time $t_m$. 


Based on the rotational form \eqref{rotform}, we consider the following discretization: 
\begin{equation}\label{Rdgscheme}
\frac{\bs n^{m+1}-\bs n^m}{\tau}=-\bs n^{m+\frac{1}{2}}\times \Big(  D_{\mathcal{F}}(\bs n)\Big|^{m+\frac{1}{2}} \times \bs n^{m+\frac{1}{2}}   \Big),
\end{equation}
where $\bs n^{m+1/2}=(\bs n^{m+1}+\bs n^m)/2$, and $D_{\mathcal{F}}(\bs n)\Big|^{m+\frac{1}{2}}$ is a discretized variational derivative that will be discussed later. 
Regardless of how the $D_{\mathcal{F}}(\bs n)\Big|^{m+\frac{1}{2}}$ is constructed, the above discretization automatically keeps the vector length.
Indeed, we take dot product by $\bs{n}^{m+1/2}$. The right-hand side vanishes due to \eqref{eq:aba}, yielding
\begin{equation}
  |\bs{n}^{m+1}|^2-|\bs{n}^m|^2=0. 
\end{equation}
Therefore, we could focus on $D_{\mathcal{F}}(\bs n)\Big|^{m+\frac{1}{2}}$ targeting the energy dissipation. 
We consider the discretized variational derivative satisfying the following relation. 
\begin{equation}\label{eq: engdis0}
\int_{\Omega} D_{\mathcal{F}}(\bs n)\Big|^{m+\frac{1}{2}}\cdot (\bs n^{m+1}-\bs n^m) dV\ge\mathcal{F}[\bs n^{m+1}]-\mathcal{F}[\bs n^{m}].
\end{equation}

\begin{thm}\label{Rdgthm}
The scheme \eqref{Rdgscheme} satisfying \eqref{eq: engdis0} is length-preserving and energy stable, i.e. for $m=1,2,\cdots$, it holds
\begin{itemize}
\item[(1)] $|\bs n^{m+1} |=|\bs n^m |,\;\;$ 
\item[(2)] $\mathcal{F}[\bs n^{m+1}]-\mathcal{F}[\bs n^{m}]\le -\tau \Big \|   D_{\mathcal{F}}(\bs n)\big|^{m+\frac{1}{2}} \times \bs n^{m+\frac{1}{2}}    \Big \|^2 $.
\end{itemize}
\end{thm}
\begin{proof}
The first property has been shown above. 
Using \eqref{eq: engdis0} and \eqref{eq: eleid}, we take the inner product between \eqref{rotform} and  $D_{\mathcal{F}}(\bs n)\Big|^{m+\frac{1}{2}}$, which directly leads to the desired energy dissipation law.
\end{proof}

\begin{rem}
  Theorem \ref{Rdgthm} does not rely on the specific form of the energy. It holds for general gradient flows constrained on a sphere. 
\end{rem}

What is remaining is to construct a $D_{\mathcal{F}}(\bs n)\Big|^{m+\frac{1}{2}}$ satisfying \eqref{eq: engdis0}. 
Since we no longer need to care about the constraint, the techniques for energy dissipation schemes can be adopted. 
The form of \eqref{Rdgscheme} indicates that the scheme is already nonlinear and could be second-order accurate in time. 
Therefore, we shall aim to obtain a second-order approximation $D_{\mathcal{F}}(\bs n)\Big|^{m+\frac{1}{2}}$ such that the equality holds for \eqref{eq: engdis0}, i.e. 
\begin{equation}\label{eq: DDprop}
\int_{\Omega} D_{\mathcal{F}}(\bs n)\Big|^{m+\frac{1}{2}}\cdot (\bs n^{m+1}-\bs n^m) dV=\mathcal{F}[\bs n^{m+1}]-\mathcal{F}[\bs n^{m}].
\end{equation}
To this end, the discrete gradient technique is the most suitable approach, and we would like to call it the rotational discrete gradient (Rdg) scheme. 

\subsection{Construction of discrete gradients}
In what follows, we consider three different kinds of discrete gradients, which satisfy the energy difference relation \eqref{eq: DDprop}.

\subsubsection{Common approaches}
There are two well-known approaches in the literature. 
The first is the mean-value discrete gradient \cite{harten1983upstream,celledoni2012preserving}:
\begin{equation}\label{eq:mv}
D^M_{\mathcal{F}}(\bs n)\Big|^{m+\frac{1}{2}}=\int_0^1 \frac{\delta \mathcal{F}}{\delta \bs n}\big[ (1-s)\bs n^{m+1}+s\bs n^m \big]ds.
\end{equation}
The relation \eqref{eq: DDprop} is deduced by
\begin{equation*}
\begin{aligned}
&\int_{\Omega}\int_0^1 \frac{\delta \mathcal{F}}{\delta \bs n}\big[ (1-s)\bs n^{m+1}+s\bs n^m \big]\cdot (\bs n^{m+1}-\bs n^m)ds\, dV\\
&=-\int_{\Omega}\int_0^1 \frac{d \mathcal{F}}{ds}\big[ (1-s)\bs n^{m+1}+s\bs n^m \big]ds \,dV=\mathcal{F}[\bs n^{m+1}]-\mathcal{F}[\bs n^{m}].
\end{aligned}
\end{equation*}
The implementation of the mean-value discrete gradient relies on numerical integration. For instance, one can adopt the Gauss quadrature as follows
\begin{equation}
D^M_{\mathcal{F}}(\bs n)\Big|^{m+\frac{1}{2}}\approx \frac{1}{2}\sum_{k=1}^{N_g} w_k  \frac{\delta \mathcal{F}}{\delta \bs n}\Big[ \frac{1}{2}(\bs n^m +\bs n^{m+1}) +\frac{\xi_k}{2} (\bs n^m-\bs n^{m+1}) \Big],
\end{equation}
where $\{\xi_k,w_k \}_{k=1}^{N_g}$ are the Legendre-Gauss points and weights, respectively. Thus, in practice, the mean-value discrete gradient preserves the discrete energy dissipation law up to a Gauss numerical quadrature error. In addition, the evaluation of numerical integration requires calculating $N_g$ energy variations. Even if we choose $N_g=2$, it results in a substantial increase in the computational cost in each iteration when we solve the nonlinear scheme. This deficiency will be demonstrated clearly through numerical experiments in Section \ref{sect: num}. 

The second is the Gonzalez discrete gradient \cite{gonzalez2000time}:
\begin{equation}\label{eq: Gondg}
D^G_{\mathcal{F}}(\bs n)\Big|^{m+\frac{1}{2}}=\frac{ \mathcal{F}[\bs n^{m+1}]- \mathcal{F}[\bs n^{m}]-\int_{\Omega} \frac{\delta \mathcal{F}}{\delta \bs n}[\bs n^{m+\frac{1}{2}}] \cdot (\bs n^{m+1}-\bs n^m) dV  }{\int_{\Omega} \big(\bs n^{m+1}-\bs n^m\big) \cdot \big(\bs n^{m+1}-\bs n^m \big)dV }(\bs n^{m+1}-\bs n^m)+\frac{\delta \mathcal{F}}{\delta \bs n}[\bs n^{m+\frac{1}{2}}].
\end{equation}
It is straightforward to verify that
\begin{equation*}
\begin{aligned}
&\int_{\Omega}D^G_{\mathcal{F}}(\bs n)\Big|^{m+\frac{1}{2}}\cdot(\bs{n}^{m+1}-\bs{n}^m)dV\nonumber\\
&=\frac{ \mathcal{F}[\bs n^{m+1}]- \mathcal{F}[\bs n^{m}]-\int_{\Omega} \frac{\delta \mathcal{F}}{\delta \bs n}[\bs n^{m+\frac{1}{2}}] \cdot (\bs n^{m+1}-\bs n^m) dV  }{\int_{\Omega} \big|\bs n^{m+1}-\bs n^m\big|^2dV }\int_{\Omega}|\bs n^{m+1}-\bs n^m|^2 dV\nonumber\\
&\;\;\;\;+\int_{\Omega}\frac{\delta \mathcal{F}}{\delta \bs n}[\bs n^{m+\frac{1}{2}}]\cdot (\bs n^{m+1}-\bs n^m)dV\\
&=\mathcal{F}[\bs n^{m+1}]-\mathcal{F}[\bs n^{m}]. 
\end{aligned}
\end{equation*}
For the Gonzalez's discrete gradient, the computational cost is lower than the mean-value discrete gradient. 
However, the Gonzalez's discrete gradient is more sensitive to roundoff errors and suffers from convergence issue, as will be shown in Section \ref{sect: num}. Actually, one can aware that when the gradient flow system tends to equilibrium, the term $\int_{\Omega} \big(\bs n^{m+1}-\bs n^m \big) \cdot \big(\bs n^{m+1}-\bs n^m \big)dV$ on the denominator of equation \eqref{eq: Gondg} may become rather small, causing severe loss of significant digits. Practically, one often modifies equation \eqref{eq: Gondg} by
\begin{equation}\label{eq: Gondgnew}
D^G_{\mathcal{F}}(\bs n)\Big|^{m+\frac{1}{2}}=\frac{ \mathcal{F}[\bs n^{m+1}]- \mathcal{F}[\bs n^{m}]-\int_{\Omega} \frac{\delta \mathcal{F}}{\delta \bs n}[\bs n^{m+\frac{1}{2}}] \cdot (\bs n^{m+1}-\bs n^m) dV  }{\int_{\Omega} \big(\bs n^{m+1}-\bs n^m\big) \cdot \big(\bs n^{m+1}-\bs n^m \big)dV+\varepsilon_0 }(\bs n^{m+1}-\bs n^m)+\frac{\delta \mathcal{F}}{\delta \bs n}[\bs n^{m+\frac{1}{2}}].
\end{equation}
with $\varepsilon_0$ a user-specified small value.


\subsubsection{The Oseen-Frank discrete gradient}
We recommend using a discrete gradient deduced from the specific form of the Oseen--Frank energy. As indicated by numerical tests in Section \ref{sect: ppt}, the proposed discrete gradient outperforms the previous two candidates in reducing the computational cost and enhancing the length preservation. 
To facilitate the derivation below, we write down the integration by parts involving the curl operator, 
\begin{equation}\label{eq: cross_intparts}
  \int_{\Omega} \bs{u}\cdot\nabla\times\bs{v}dV=\int_{\partial\Omega} \bs{u}\cdot\bs{\nu}\times\bs{v}dS +\int_{\Omega} \bs{v}\cdot\nabla\times\bs{u}dV.
\end{equation}
where we use $\bs \nu$ to represent the outward normal unit vector. 
To fix the idea, we assume periodic boundary conditions so that the last term in \eqref{eq: osfrank} does not appear and surface integrals vanish when doing integration by parts.
In this case, the variational derivative is given by
\begin{equation}\label{variation}
\begin{aligned}
\frac{\delta \mathcal{F}[\bs n]}{\delta \bs n}& =-k_{1}\nabla(\nabla \cdot \bs{n})+k_{2}\Big(   (\bs{n}\cdot \nabla \times \bs{n})(\nabla \times \bs{n})+   \nabla\times\big((\bs{n}\cdot \nabla \times \bs{ n})\bs{n}  \big)        \Big)\\
&+k_{3}\Big(  (\nabla \times \bs{n})\times (\bs{n}\times \nabla \times \bs{n})+   \nabla\times\big((\bs{n}\times \nabla \times \bs{ n})\times\bs{n}  \big) \Big).
\end{aligned}
\end{equation}

\begin{prop}
The Oseen-Frank discrete gradient takes the form
\begin{equation}\label{eq:OFDG}
\begin{aligned}
D^O_{\mathcal{F}}(\bs n)\Big|^{m+\frac{1}{2}}=&-k_1\nabla (\nabla \cdot \bs n^{m+1/2}) +{k_2}\Big\{ \beta^{m+1/2}  \nabla \times \bs n^{m+1/2} +\nabla \times (\beta^{m+1/2}  \bs n^{m+1/2}) \Big\}\\
&+{k_3}\Big\{   (\nabla \times \bs{n}^{m+1/2})\times \bs \omega^{m+1/2}+   \nabla\times\big(\bs \omega^{m+1/2}\times\bs{n}^{m+1/2}  \big)   \Big\},
\end{aligned}
\end{equation}
where $\bs \omega^{m+1/2}$ and $\beta^{m+1/2}$ are defined as
\begin{equation}\label{eq: omega}
\begin{split}
& \bs \omega^{m+1/2}=\big(\bs n^{m+1}\times \nabla \times   \bs n^{m+1}+\bs n^{m}\times \nabla \times \bs n^{m} \big)/2,\\
&  \beta^{m+1/2}=\big(\bs n^{m+1}\cdot \nabla \times \bs n^{m+1}+\bs n^{m}\cdot \nabla \times \bs n^{m} \big)/2.
\end{split}
\end{equation}
It is a second-order approximation of the Oseen-Frank energy variation \eqref{variation} and satisfies the discrete energy difference relation
\begin{equation}\label{eq: DDpropof}
\int_{\Omega} D^O_{\mathcal{F}}(\bs n)\Big|^{m+\frac{1}{2}}\cdot (\bs n^{m+1}-\bs n^m) dV=\mathcal{F}[\bs n^{m+1}]-\mathcal{F}[\bs n^{m}].
\end{equation}
\end{prop}

\begin{proof}
Let us first construct the proposed discrete gradient and show that it satisfies the discrete energy relation \eqref{eq: DDprop}. Neglecting the boundary terms, one splits the Oseen-Frank energy functional \eqref{eq: osfrank} into three parts
\begin{equation}
\mathcal{F}[\bs n]=\frac{k_1}{2}\mathcal{F}_1[\bs n]+\frac{k_2}{2}\mathcal{F}_2[\bs n]  +\frac{k_3}{2}\mathcal{F}_3[\bs n]
\end{equation}
with
\begin{equation}
\begin{aligned}
\mathcal{F}_1[\bs n]=\int_{\Omega} (\nabla \cdot \bs{n})^{2}dV,\quad \mathcal{F}_2[\bs n]=(\bs{n}\cdot \nabla \times \bs{n})^{2}dV,\quad \mathcal{F}_3[\bs n]=(\bs{n}\times \nabla \times \bs{n})^{2}dV.
\end{aligned}
\end{equation}

For the first term, it is straightforward to derive that 
\begin{equation*}
\begin{aligned}
&\mathcal{F}_1[\bs n^{m+1}] -\mathcal{F}_1[\bs n^m]=\int (\nabla \cdot  \bs{n}^{m+1}+\nabla \cdot  \bs{n}^{m})(\nabla \cdot  \bs{n}^{m+1}-\nabla \cdot  \bs{n}^{m}) dV\\
=&2 \int \nabla \cdot ( \bs{n}^{m+1/2}) \nabla \cdot  (\bs{n}^{m+1}-\bs{n}^{m}) dV=\int -2\nabla ( \nabla \cdot \bs n^{m+1/2})\cdot (\bs n^{m+1}-\bs n^m) dV, 
\end{aligned}
\end{equation*}
where integration by parts is done in the last equality. Let us define
\begin{equation}\label{eq: df1}
D_{\mathcal{F}_{1}}(\bs n)\Big|^{m+\frac{1}{2}}=-2\nabla ( \nabla \cdot \bs n^{m+1/2})
\end{equation}
and one readily observes that 
\begin{equation}\label{eq: df1rela}
\int D_{\mathcal{F}_{1}}(\bs n)\Big|^{m+\frac{1}{2}}\cdot (\bs n^{m+1}-\bs n^m)  dV=\mathcal{F}_1[\bs n^{m+1}] -\mathcal{F}_1[\bs n^m].
\end{equation}

For the second term, one has the identity
\begin{equation*}
\begin{aligned}
& |\bs{n}^{m+1}\cdot \nabla \times \bs{n}^{m+1}|^{2}- |\bs{n}^{m}\cdot \nabla \times \bs{n}^{m}|^{2}\\
&=2\Big((\bs{n}^{m+1}-\bs{n}^{m})\cdot \nabla \times \bs{n}^{m+1/2} +\bs{n}^{m+1/2}\cdot \nabla\times (\bs{n}^{m+1}-\bs{n}^{m})            \Big)\beta^{m+1/2},
 \end{aligned}
\end{equation*}
where $\beta^{m+1/2}$ is defined in equation \eqref{eq: omega}. 
Thus, using \eqref{eq: cross_intparts}, we deduce that
\begin{equation*}
\begin{aligned}
\mathcal{F}_2[\bs n^{m+1}] -\mathcal{F}_2[\bs n^m]=2 \int \Big\{\beta^{m+1/2}  \nabla \times \bs n^{m+1/2} +\nabla \times (\beta^{m+1/2}  \bs n^{m+1/2}) \Big\} \cdot (\bs n^{m+1}-\bs n^m) dV.
\end{aligned}
\end{equation*}
Similarly, one can define the associated discrete gradient for the second term
\begin{equation}\label{eq: df2}
D_{\mathcal{F}_{2}}(\bs n)\Big|^{m+1/2}=2\Big\{ \beta^{m+1/2}  \nabla \times \bs n^{m+1/2} +\nabla \times (\beta^{m+1/2}  \bs n^{m+1/2}) \Big\}.
\end{equation}

For the third term, we begin with
\begin{align*}
  |\bs{n}^{m+1}\times\nabla\times\bs{n}^{m+1}|^2-|\bs{n}^m\times\nabla\times\bs{n}^m|^2
  =\big(\bs{n}^{m+1}\times\nabla\times\bs{n}^{m+1}-\bs{n}^m\times\nabla\times\bs{n}^m\big)\cdot 2\bs{\omega}^{m+1/2}. 
\end{align*}
We express
\begin{align*}
  \bs{n}^{m+1}\times\nabla\times\bs{n}^{m+1}-\bs{n}^m \times\nabla\times\bs{n}^m
  =(\bs{n}^{m+1}-\bs{n}^m)\times\nabla\times\bs{n}^{m+1/2}+\bs{n}^{m+1/2}\times\nabla\times(\bs{n}^{m+1}-\bs{n}^m). 
\end{align*}
Then we use \eqref{eq: eleid} and \eqref{eq: cross_intparts} to obtain 
\begin{align*}
  &\mathcal{F}_3[\bs n^{m+1}] -\mathcal{F}_3[\bs n^m]\nonumber\\
  =\,&2\int_{\Omega} \big((\nabla\times\bs{n}^{m+1/2})\times\bs{\omega}^{m+1/2}\big)\cdot(\bs{n}^{m+1}-\bs{n}^m)+(\bs{\omega}^{m+1/2}\times\bs{n}^{m+1/2})\cdot\nabla\times(\bs{n}^{m+1}-\bs{n}^m)dV\nonumber\\
  =\,&2\int_{\Omega} \big((\nabla\times\bs{n}^{m+1/2})\times\bs{\omega}^{m+1/2}\big)\cdot(\bs{n}^{m+1}-\bs{n}^m)+\nabla\times(\bs{\omega}^{m+1/2}\times\bs{n}^{m+1/2})\cdot(\bs{n}^{m+1}-\bs{n}^m)dV.\label{eq: df3rela}
\end{align*}
Thus, we can define
\begin{equation}\label{eq: df3}
D_{\mathcal{F}_{3}}(\bs n)\Big|^{m+1/2}=2 \Big\{ (\nabla \times \bs{n}^{m+1/2})\times \bs \omega^{m+1/2}+   \nabla\times\big(\bs \omega^{m+1/2}\times\bs{n}^{m+1/2}  \big)  \Big\}.
\end{equation}

Combing equations \eqref{eq: df1}, \eqref{eq: df2} and \eqref{eq: df3}, one arrives at the proposed Oseen-Frank discrete gradient \eqref{eq:OFDG}-\eqref{eq: omega} and the discrete energy relation \eqref{eq: DDprop}.

Notice that the form of the discrete gradient resembles that of the energy variation \eqref{variation} at time step $m+1/2$. 
In particular, the terms $\bs{n}\cdot \nabla \times \bs{n}$ and $\bs{n}\times \nabla \times \bs{n}$ are replaced by $\beta^{m+1/2}$ and $\bs \omega^{m+1/2}$ in equation \eqref{eq: omega}, respectively, which follows the trapezoid formula. 
%
%
%
%
%
%
Thus, it is straightforward to see that the proposed Oseen-Frank discrete gradient is of second-order accuracy of the continuous energy variation \eqref{variation}.
\end{proof}

\begin{rem}
  The derivation also applies for Dirichlet boundary conditions or natural boundary conditions (see \eqref{ntrbnd}), which we explain briefly in Appendix. 
  \end{rem}

\begin{rem}
  We only present time discretization. For spatial discretization, we use the Fourier spectral method since we adopt periodic boundary conditions.
  For boundary conditions of other types, it has no essential difficulty to incorporate spatial discretizations of Galerkin type with Lagrange basis at discrete points.
  In this case, the preservation of length constraint is guaranteed at these discrete points. 
\end{rem}

\subsection{Time adaptive solution algorithm} \label{sect: adap}
The proposed Rdg scheme \eqref{Rdgscheme} is nonlinear, we can solve the resultant nonlinear system efficiently by inexact Newton-Krylov (INK) method combined with Armijo line search technique (see e.g. \cite{kelley1995iterative}).  It is worthwhile to point out that though the Rdg scheme is unconditionally energy stable, a large time step size may slow down the INK solution algorithm.
Moreover, numerical experiments in the next section demonstrate that the vector field often exhibits multi-stage evolution phenomenon. It may evolve quickly in transient stage and remain almost unchanged for the rest of the simulation time. Thus, it is highly desirable to adjust the time step sizes adaptively for efficient and accurate simulations. Inspired by \cite{qiao2011adaptive}, we propose the following time-adaptivity strategy
\begin{equation}\label{eq: adap}
\tau_{m+1}={\rm max}\bigg(\tau_{\rm min},  \frac{\tau_{\rm max}}{ \sqrt{1+\alpha| (\mathcal{F}[\bs n^{m}]-\mathcal{F}[\bs n^{m-1}])/\tau_{m}        |^2 }}   \bigg),
\end{equation} 
where $\tau_{m}$ is the $m$-th time step, $\alpha$ is a constant chosen to adjust the speed of change of the step size, and $\tau_{\rm max}$ and $\tau_{\rm min}$ are the upper and lower bounds of the time step sizes, respectively.

\section{Numerical experiments}\label{sect: num}
In this section, we conduct several numerical experiments to demonstrate the accuracy, efficiency and robustness of the proposed rotational discrete gradient method for the Oseen-Frank gradient flows, with particular emphasis on its properties of energy stability and exact length preservation of the vector field.
A manufactured solution is firstly employed to show the spatial and temporal convergence rates of this method with the mean-value discrete gradient \eqref{eq:mv} (dubbed as "MvRdg" method), the Gonzalez's discrete gradient \eqref{eq: Gondg} (dubbed as "GonRdg" method) and the proposed Oseen-Frank discrete gradient \eqref{eq:OFDG} (dubbed as "OFRdg" method).
In MvRdg, we set the number of the Gauss quadrature points $N_g=4$ in order to acquire adequate accuracy. 
The performance of the rotational discrete gradient method with these three different discrete gradients are compared  to show the superiority of the proposed OFRdg method on accuracy and robustness.
The efficiency of the adaptive time stepping strategy is also demonstrated using simulations with various initial conditions and elastic coefficients $(k_1,k_2,k_3)$.
Moreover, we investigate the influence of the elastic constants on the dynamics of vector field, where multi-stage evolution phenomena of the vector field can be observed, especially when the magnitudes of the elastic coefficients are highly disparate. 

\subsection{Convergence test}
We first use a manufactured solution to the Oseen-Frank gradient flow system \eqref{rotform}   to numerically demonstrate the rates of convergence in space and time of the Rdg scheme proposed in Section \ref{sect: rdg}. Let us consider the following analytic solution $\bs{n}=(n_1,n_2,n_3)^{\intercal}$ on $\Omega=[0,2\pi]^3$, where
\begin{equation}\label{ana}
\begin{aligned}
&n_1(\bs x,t)=\sin\big(\sin(x_1+t)\cos(x_2)\sin(x_3)\big)\cos\big(\cos(x_1)\sin(x_2+t)\cos(x_3)\big),\\
&n_2(\bs x,t)=\sin\big(\sin(x_1+t)\cos(x_2)\sin(x_3)\big)\sin\big(\cos(x_1)\sin(x_2+t)\cos(x_3)\big),\\
&n_3(\bs x,t)=\cos\big(\sin(x_1+t)\cos(x_2)\sin(x_3)\big),
\end{aligned}
\end{equation}
 such that $\bs{n}$ is periodic and $|\bs{n}|=1$. A body force term $\bs f(\bs x,t)$ is chosen such that the analytic solution $\bs n$ satisfies 
 \begin{equation}\label{eq: uf}
\bs n_t=-\bs n\times \Big(   \frac{\delta \mathcal{F}[\bs n]}{\delta \bs n}\times \bs n   \Big)+\bs f.
 \end{equation} 
 The proposed MvRdg, GonRdg and OFRdg methods are employed to numerically integrate system \eqref{eq: uf} from $t=0$ to $t=t_f$. The numerical solutions are then compared with the contrived solution at $t=t_f$ and the numerical errors in $L^{\infty}$ norm are measured for  $\bs n=(n_1,n_2,n_3)^{\intercal}.$ In order to demonstrate the flexibility of the Rdg method for dealing with anisotropic elasticity, let us choose the elastic coefficients to be $(k_1,k_2,k_3)=(2,3,4)$. We set $t_0=0$, $t_f=0.2$ and adopt a uniform number of Fourier collocation points in each dimension $N_1=N_2=N_3=N$. The stopping tolerance for the INK solver is set to $10^{-8}$, if otherwise specified.  
 
In the spatial convergence test, we  fix the time step size $\tau=10^{-4}$ and vary $N$ systematically between 6 and 30. Figure \ref{figs: spatialtest} shows the $L^{\infty}$-errors of $u_1,u_2,u_3$ versus $N$. It can be observed that when $N$ is below 26, the errors computed using the Rdg method with these three discrete gradients decrease exponentially with increased $N$, displaying an
exponential convergence rate in space.  And when $N>26$, the error curves remain parallel, showing a saturation due to the temporal truncation error. One can also distinguish a slightly better resolution of the OFRdg method against that of the other two methods.  

\begin{figure}[htbp]
\begin{center}
  \subfigure[$L^{\infty}$-error of $n_1$]{ \includegraphics[scale=.22]{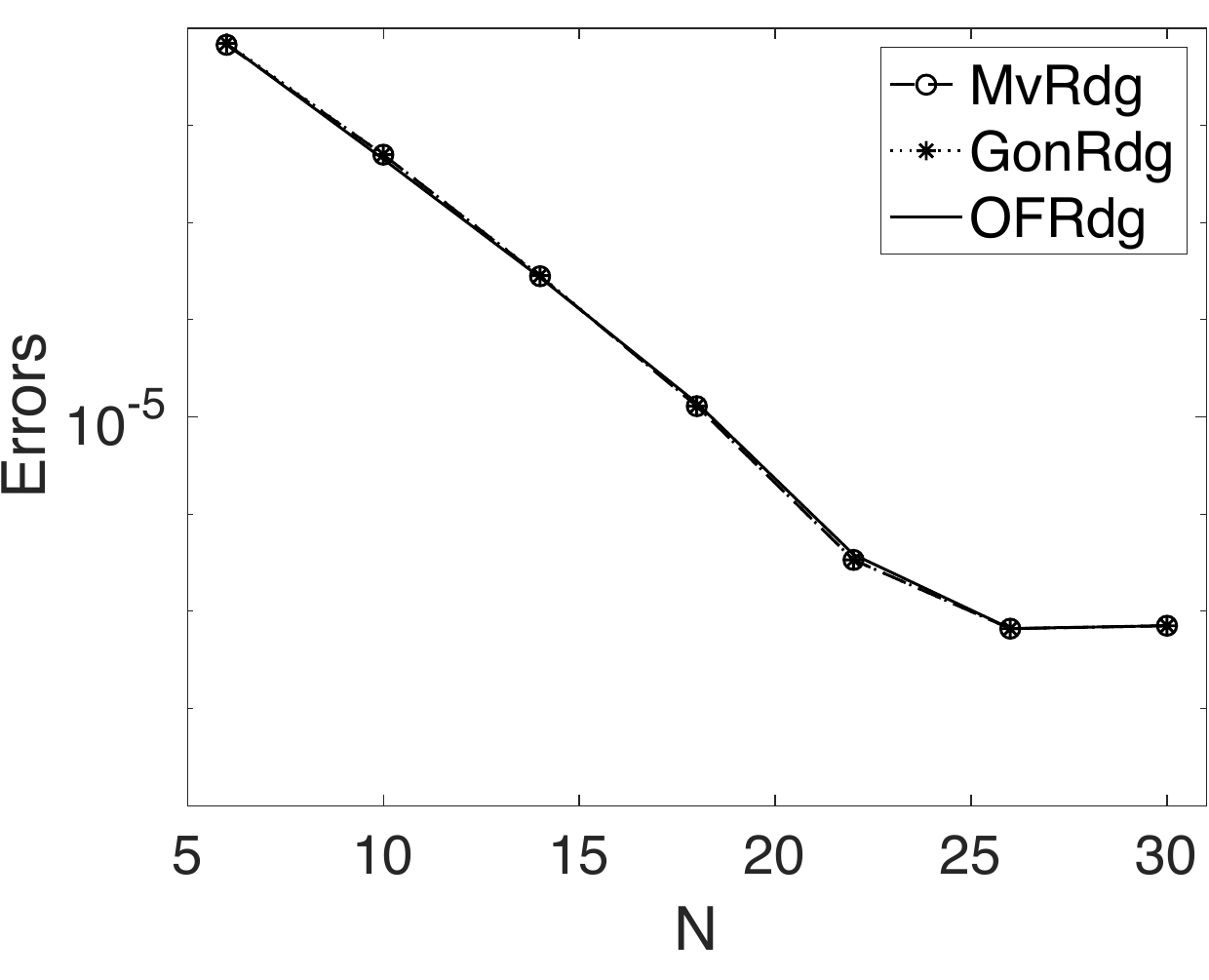}}
  \subfigure[$L^{\infty}$-error of $n_2$ ]{ \includegraphics[scale=.22]{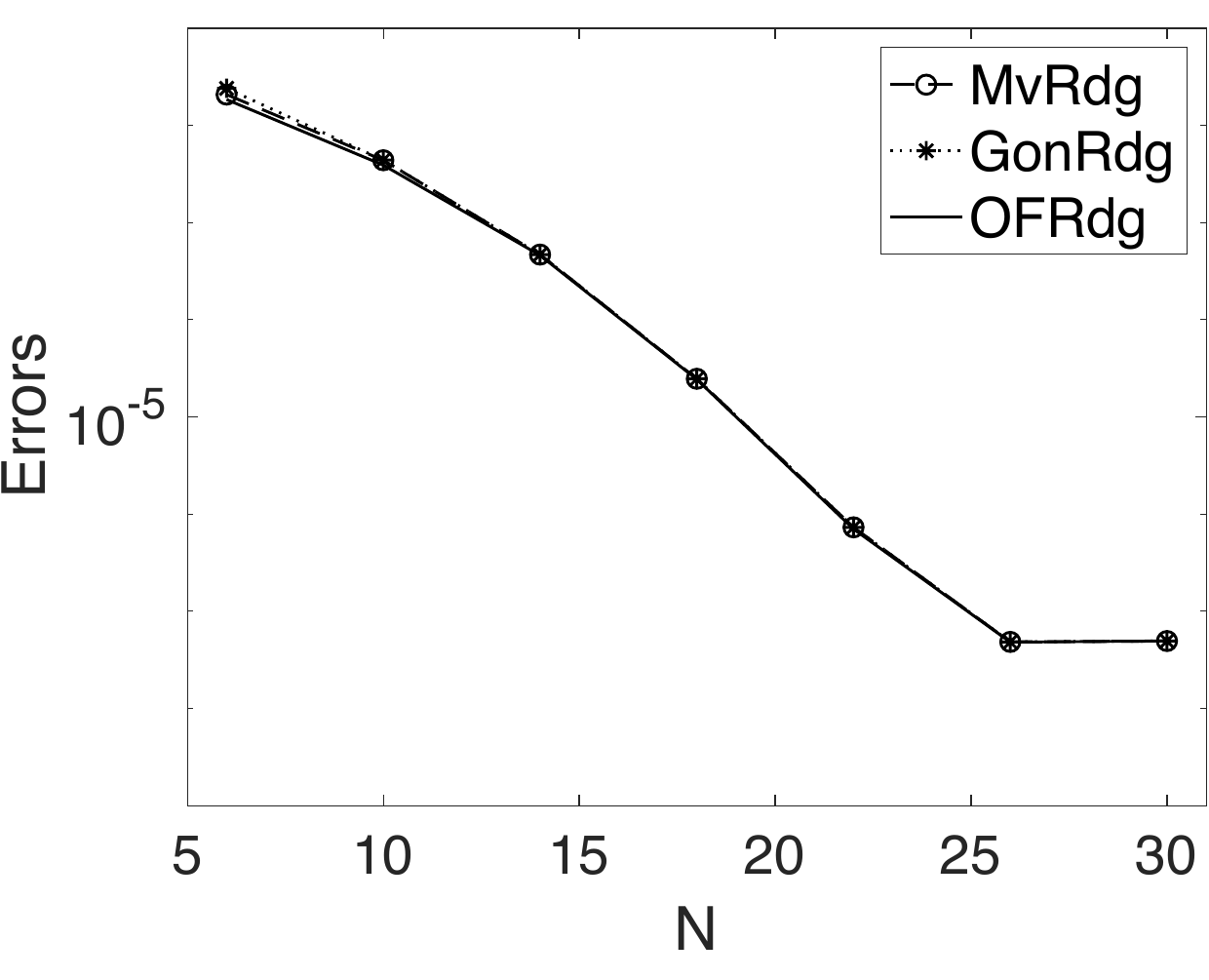}}
 \subfigure[ $L^{\infty}$-error of $n_3$]{ \includegraphics[scale=.22]{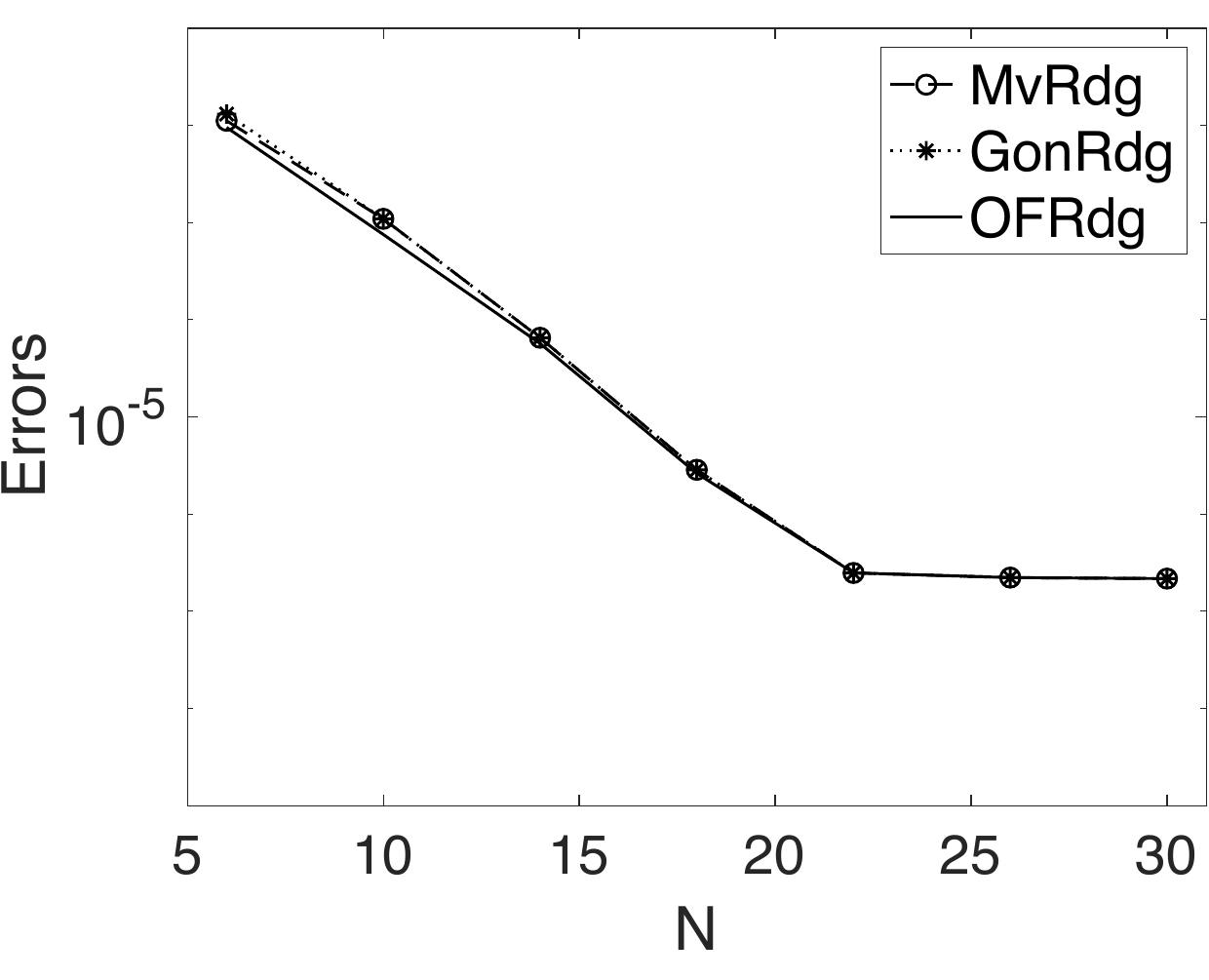}} 
    \caption{\small Spatial convergence test: $L^{\infty}$-errors of the MvRdg, GonRdg and OFRdg methods as a function of the Fourier collocation points $N_1=N_2=N_3=N$ for $n_1$ (a), $n_2$ (b) and $n_3$ (c).} 
   \label{figs: spatialtest}
\end{center}
\end{figure}

 \begin{figure}[htbp]
\begin{center}
  \subfigure[$L^{\infty}$-error of $n_1$]{ \includegraphics[scale=.22]{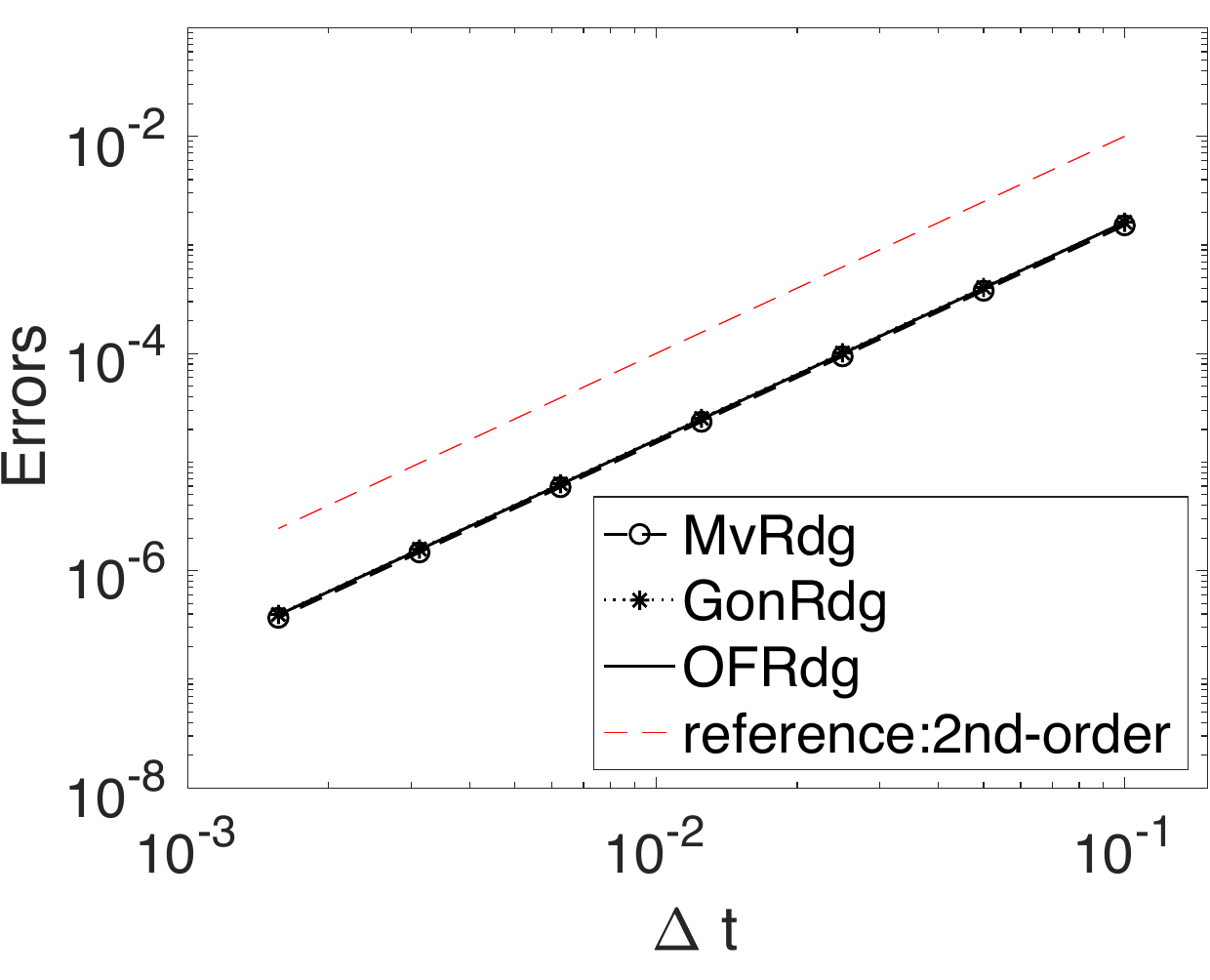}}
  \subfigure[$L^{\infty}$-error of $n_2$ ]{ \includegraphics[scale=.22]{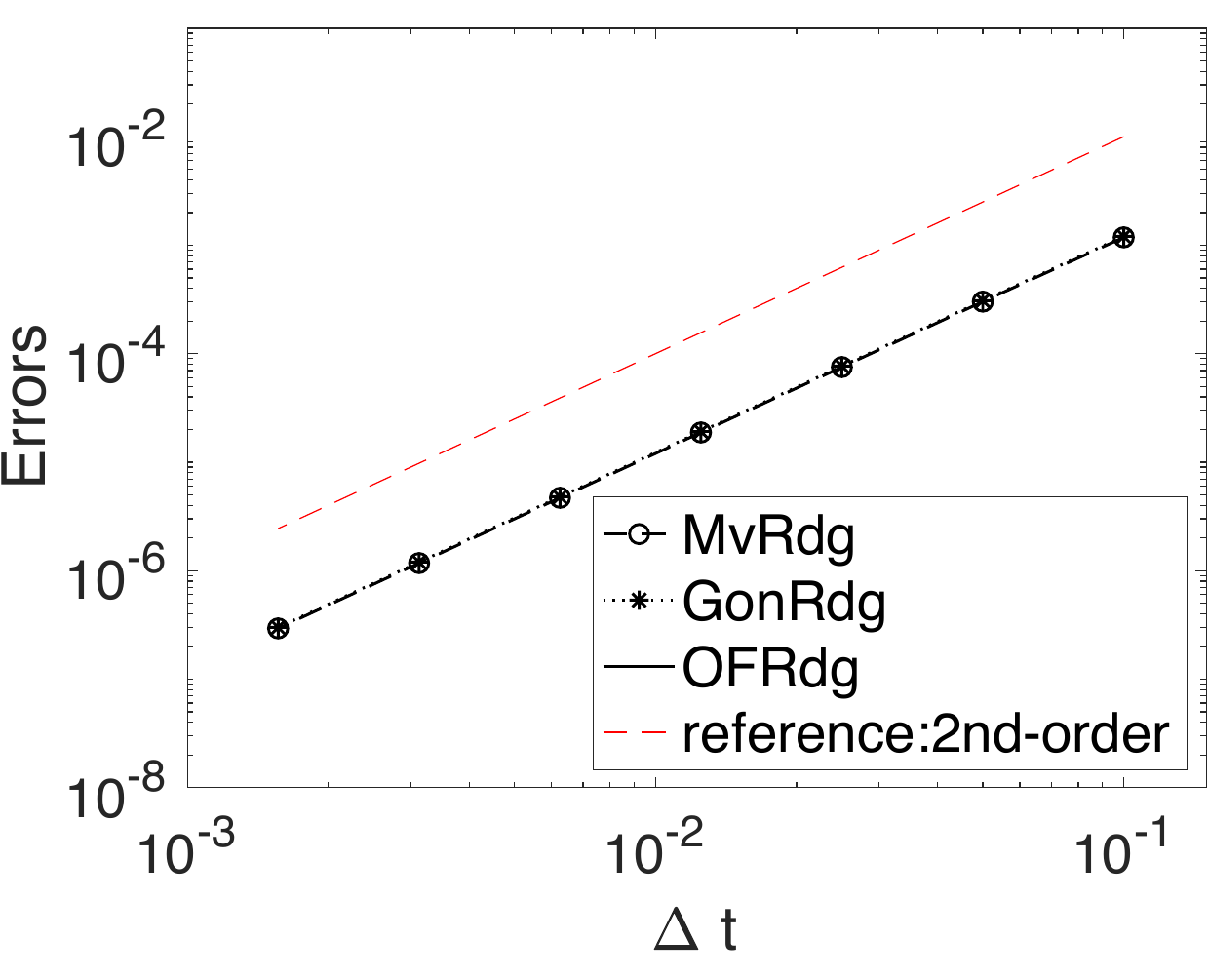}}
 \subfigure[ $L^{\infty}$-error of $n_3$]{ \includegraphics[scale=.22]{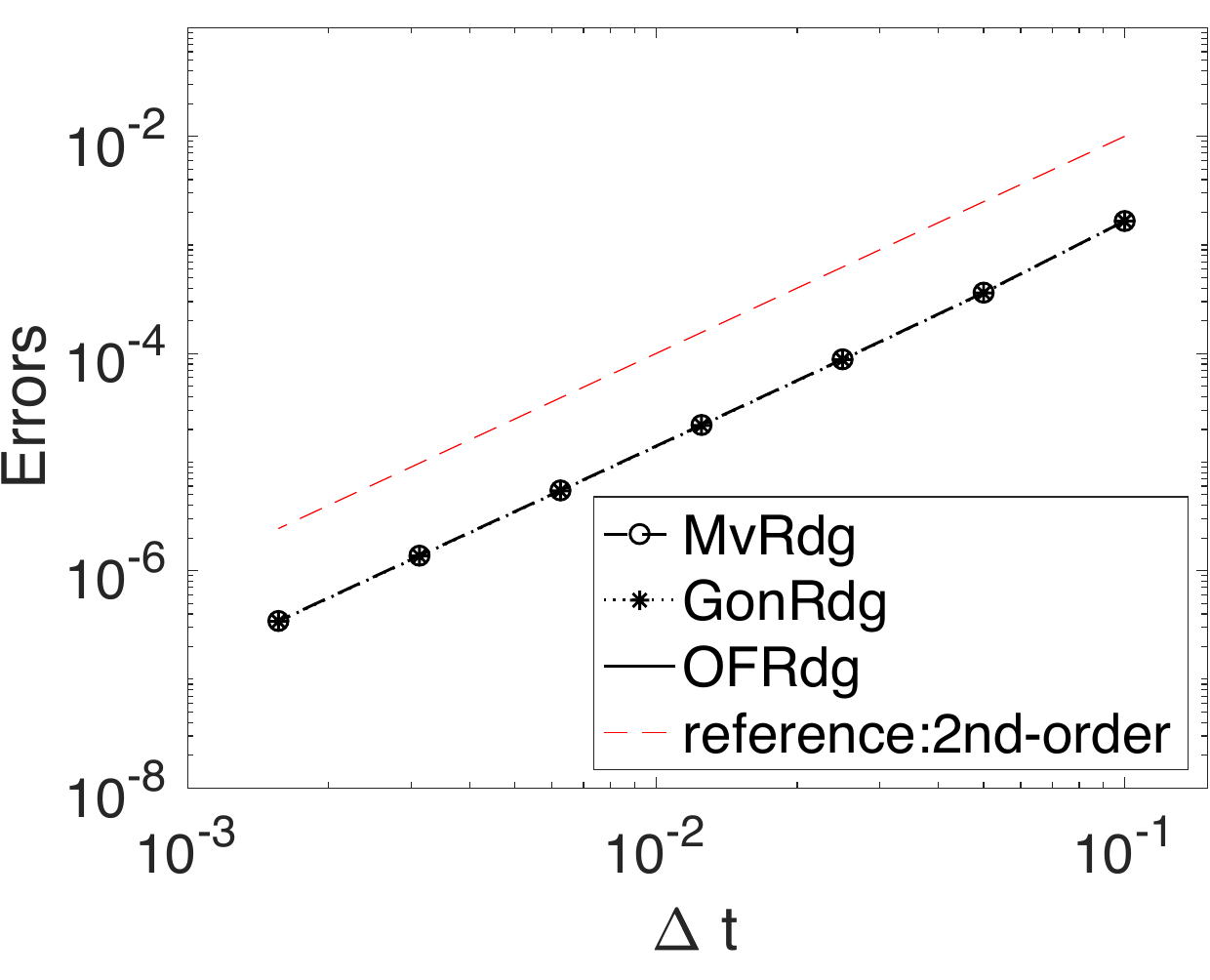}}
    \caption{\small Temporal convergence test: $L^{\infty}$-errors of the MvRdg, GonRdg and OFRdg methods as a function of the time step size $\tau$ for $n_1$ (a), $n_2$ (b) and $n_3$ (c). } 
   \label{figs: temporaltest}
\end{center}
\end{figure}

In the temporal convergence test, we fix $N$ at a large value 40 and vary the time step size $\tau$ systematically between $0.01$ and $0.0016$. Then for every $\tau$, one gets the $L^{\infty}$-error between the numerical solution and the manufactured solution computed at $t_f=0.2$. Figure \ref{figs: temporaltest} shows the numerical errors as functions of $\tau$.  A second-order convergence rate in time is observed for the Rdg method and the differences of the numerical errors obtained with the three discrete gradients are negligible.

\subsection{Property-preserving test}\label{sect: ppt}

In what follows, we conduct several numerical experiments to demonstrate that the proposed Rdg method is energy stable and length preserving. We assume that the vector field varies in the $x_1$-$x_2$ plane and is uniform along the $x_3$-axis, and prescribe the computational domain to be $\Omega=[-1,1]^{2}$. Correspondingly, we use 40 Fourier collocation points evenly in the $x_1$ and $x_2$ directions. 

We first consider the evolution of the nematic liquid crystal with the following initial distribution of the vector field
\begin{equation}\label{eq:utest1}
\bs{n}^{0}=\Big( \sin\big(2\sin(\pi x_1)\big) \cos(\pi x_2)   ,    \sin\big( 2\sin(\pi x_1) \big) \sin(\pi x_2)          ,    \cos\big( 2\sin(\pi x_1)  \big)  \Big)^{\intercal},
\end{equation}  
such that $\bs n^0$ satisfies the periodic boundary condition and is of unit length $|\bs{n}^{0}|=1$. We first set the elastic coefficients to be $k_1=k_2=k_3=1$. The Rdg method with three different discrete gradients is employed to integrate equation \eqref{rotform} from $t=0$ to $t=10$ using time step size $\tau=10^{-3}$.

\begin{figure}[htbp]
 \begin{center}  
 \subfigure[Initial profile]{  \includegraphics[scale=.42] {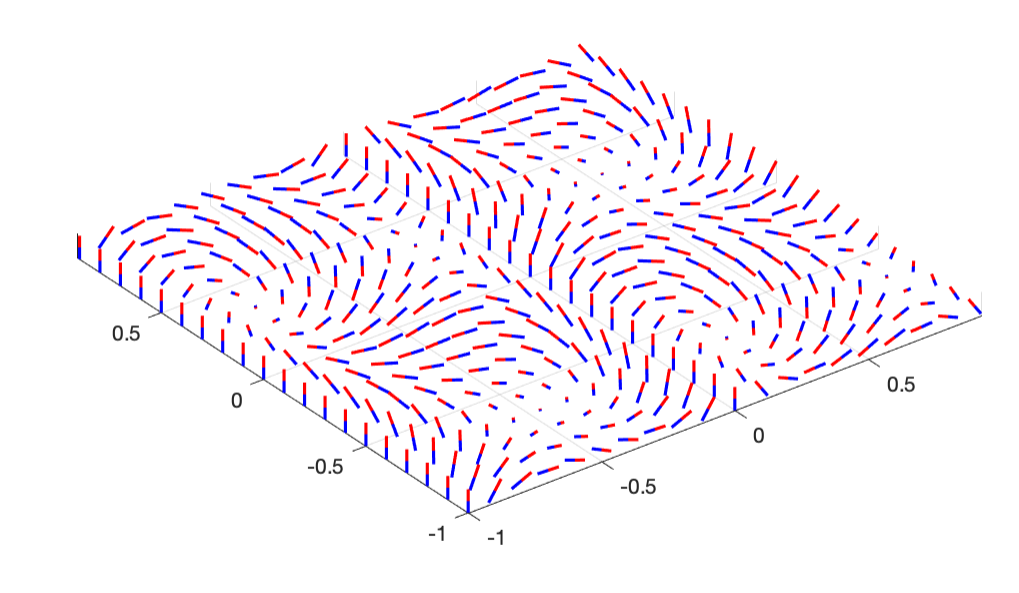}}  
 \subfigure[Energy vs Time]{  \includegraphics[scale=.28]{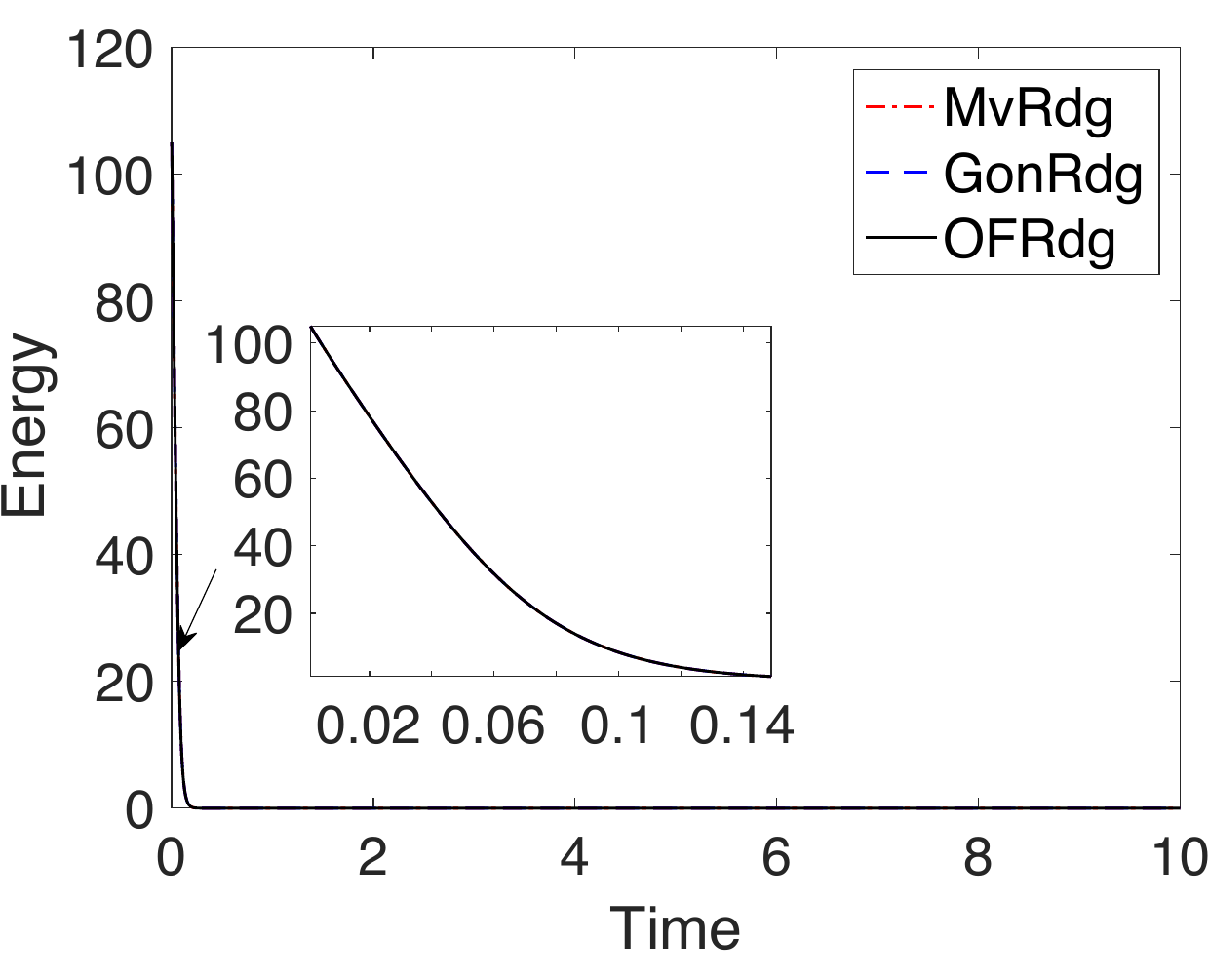}} \\
 \subfigure[Length error vs Time]{  \includegraphics[scale=.28]{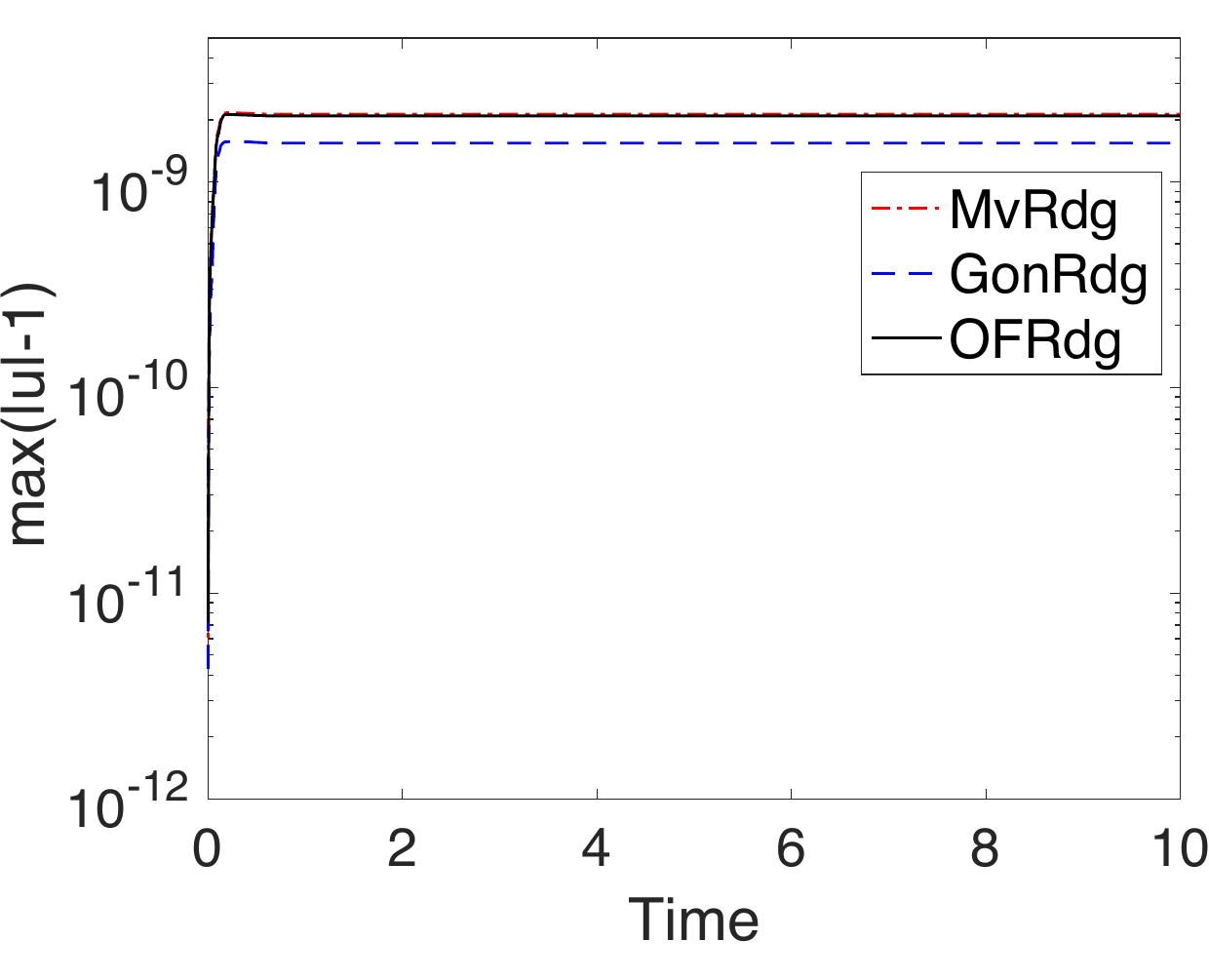}
 }\quad \quad \;
 \subfigure[Computational cost vs Time]{  \includegraphics[scale=.28]{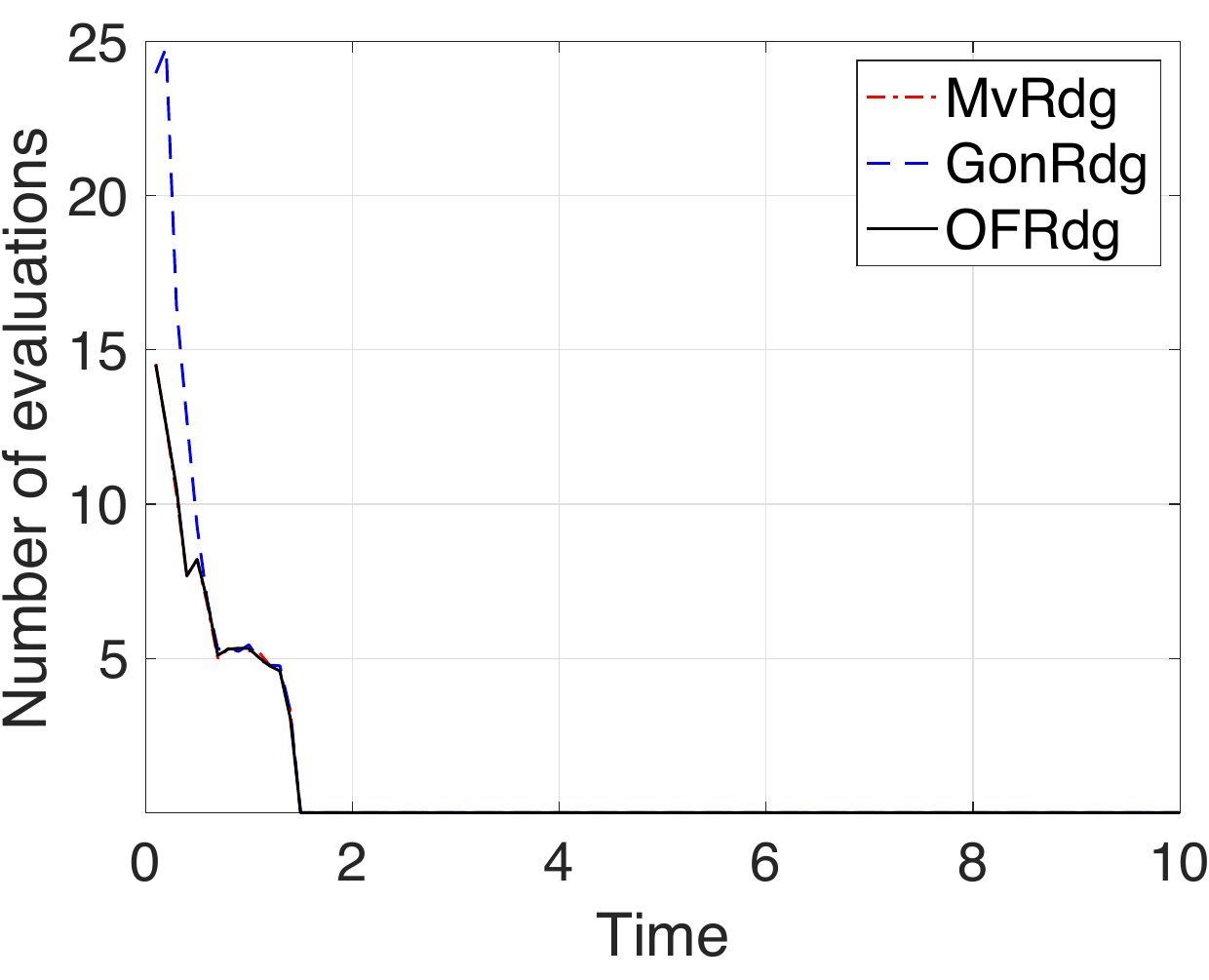}
 } 
 \caption{\small Evolution of the vector field with initial field distribution \eqref{eq:utest1}. (a) Initial profile of the vector field; (b) time history of energy; (c) time history of the length error; (d) time history of the number of function evaluations at each step.}  \label{fig: homtest1}
 \end{center} 
 \end{figure}

 \begin{figure}[htbp]
 \begin{center}  
  \subfigure[Final profile via MvRdg]{  \includegraphics[scale=.28]{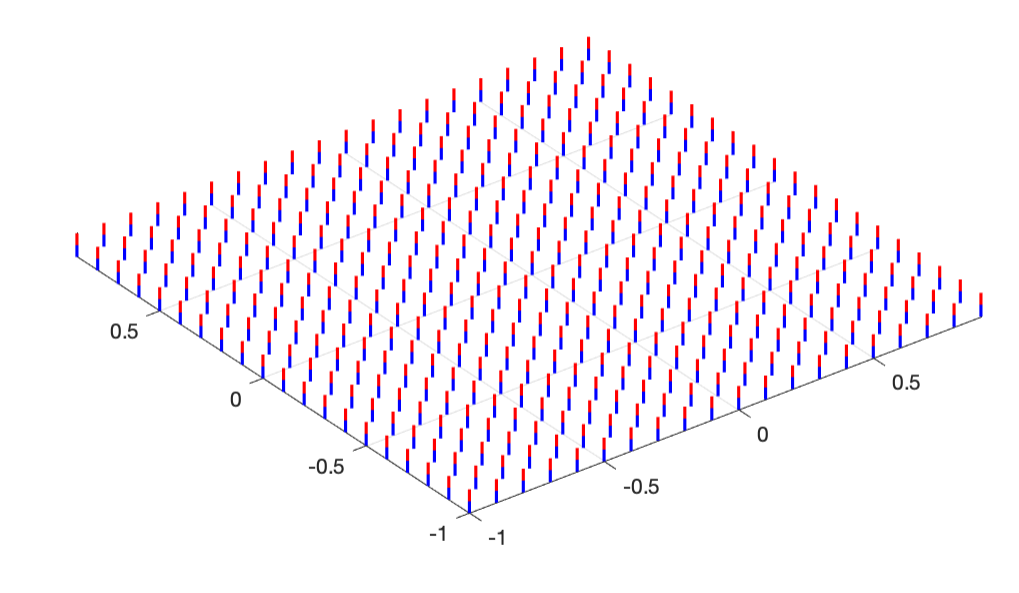}}
 \subfigure[Final profile via  GonRdg]{  \includegraphics[scale=.28]{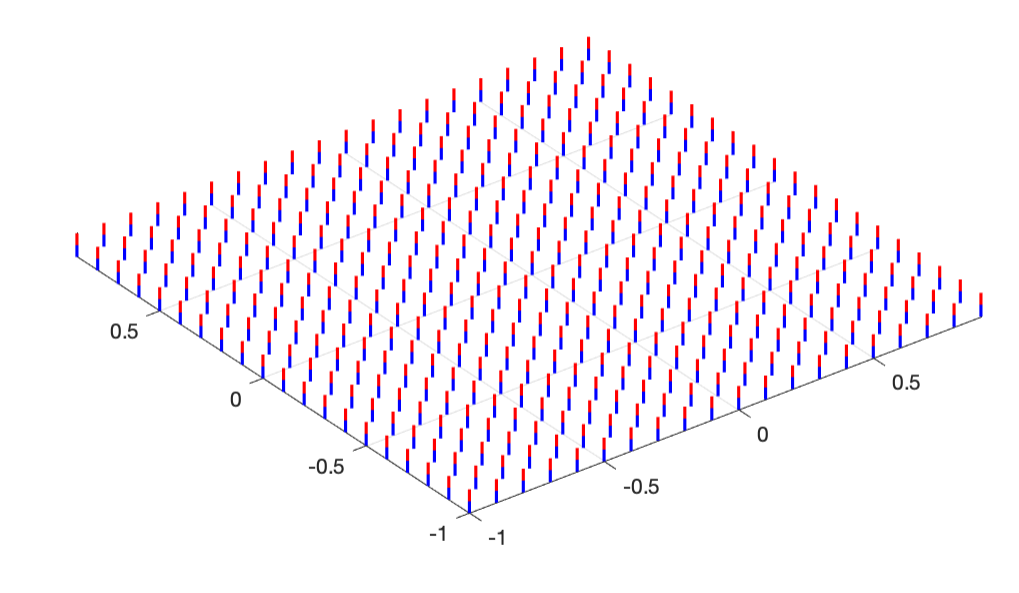}} 
  \subfigure[Final profile via OFRdg]{  \includegraphics[scale=.28]{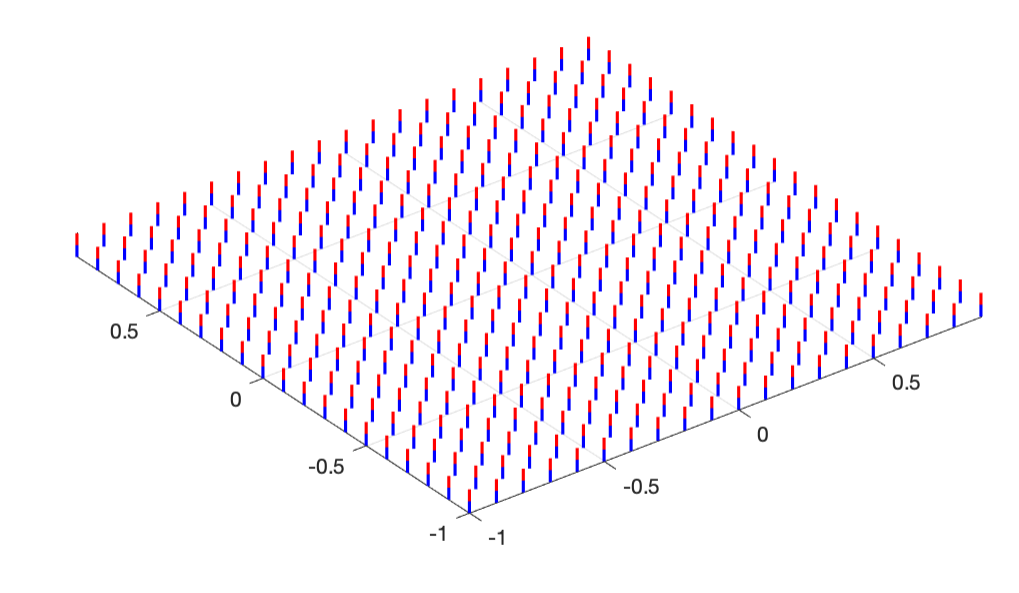}} 
 \caption{ \small The equilibrium state of the vector field with initial field distribution \eqref{eq:utest1} computed using (a) the MvRdg method; (b) GonRdg method; (c) OFRdg method.}   \label{fig: homtest1finalprofile}
 \end{center}
 \end{figure}

Figure \ref{fig: homtest1} (a) depicts a side view (azimuth $=-37.5^{\circ}$, elevation $=30^{\circ}$) of the initial molecule arrangement using colored rods with the red side represents the direction of the vector field while the blue side represents the opposite direction. Figure \ref{fig: homtest1} (b) shows the time history of the energy $\mathcal{F}[\bs n]$, computed using the MvRdg, GonRdg and OFRdg methods. It can be observed that the energy decays drastically from $t=0$ to $t=0.15$ and remains 0 afterwards. One could not discern any difference of the energy curves obtained with these three methods, even in the zoom-in subfigure. Figure \ref{fig: homtest1} (c) depicts the time history of the length error under $L^{\infty}$-norm. It can be seen that the length error curves level off at around $10^{-9}$, showing that the unit-length is preserved, up to the tolerance of the inexact Newton-Krylov (INK) solver for solving the discrete nonlinear system \eqref{Rdgscheme}.  Figure \ref{fig: homtest1} (c) plots the time history of the number of function evaluations performed in the INK solver at each time step.  Compared with the MvRdg and the OFRdg methods, one can observe that there is a remarkable increase in the number of function calls of the GonRdg method,  implying a relatively slow of convergence of the INK solver. We also record the average time consumed per time step for these three methods, which are 0.0359s, 0.0316s and 0.0158s for the MvRdg, GonRdg and the OFRdg methods, respectively. It can be seen clearly that the OFRdg method has around $2\times$ speed up compared with that of the other two methods.   In Figure \eqref{fig: homtest1finalprofile}, we plot the associated equilibrium state of the vector field computed using these three methods.  It can be seen that the vector fields all turn to homogeneous distributions and negligible difference can be observed, demonstrating an agreement in the numerical results obtained with these three methods.  

Next, we alter the initial distribution of the vector field to
\begin{equation}\label{eq:utest2}
\bs{n}^{0}=\Big(\sin\big(\pi x_1+2\cos(\pi x_2) \big),\,0\,,\cos \big(\pi x_1+2\cos(\pi x_2) \big) \Big)^{\intercal}.
\end{equation}  
In this example, the evolution of the vector field is more complex and correspondingly, the numerical simulations are more challenging. 

\begin{figure}[htbp]
 \begin{center}  
 \subfigure[Initial profile]{  \includegraphics[scale=.42]{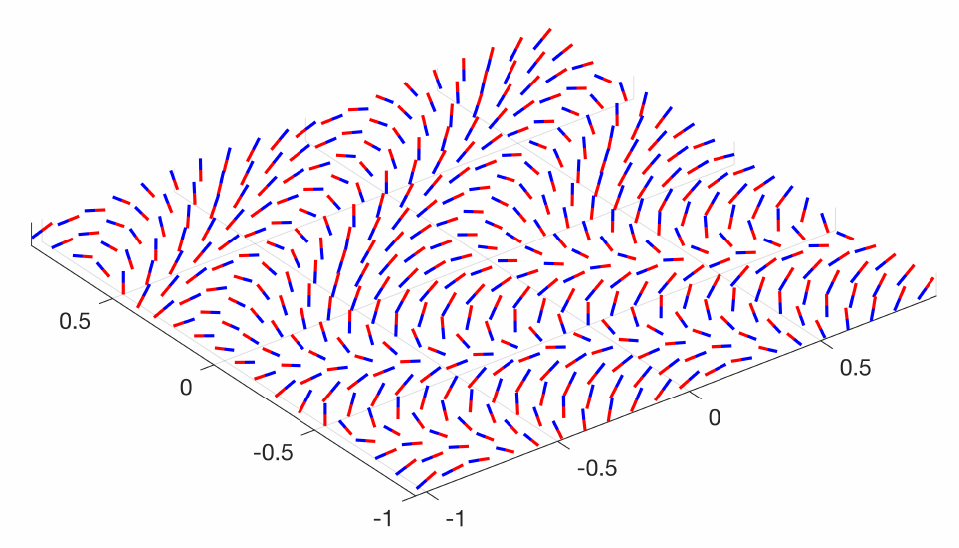}} 
 \subfigure[Energy vs Time]{  \includegraphics[scale=.28]{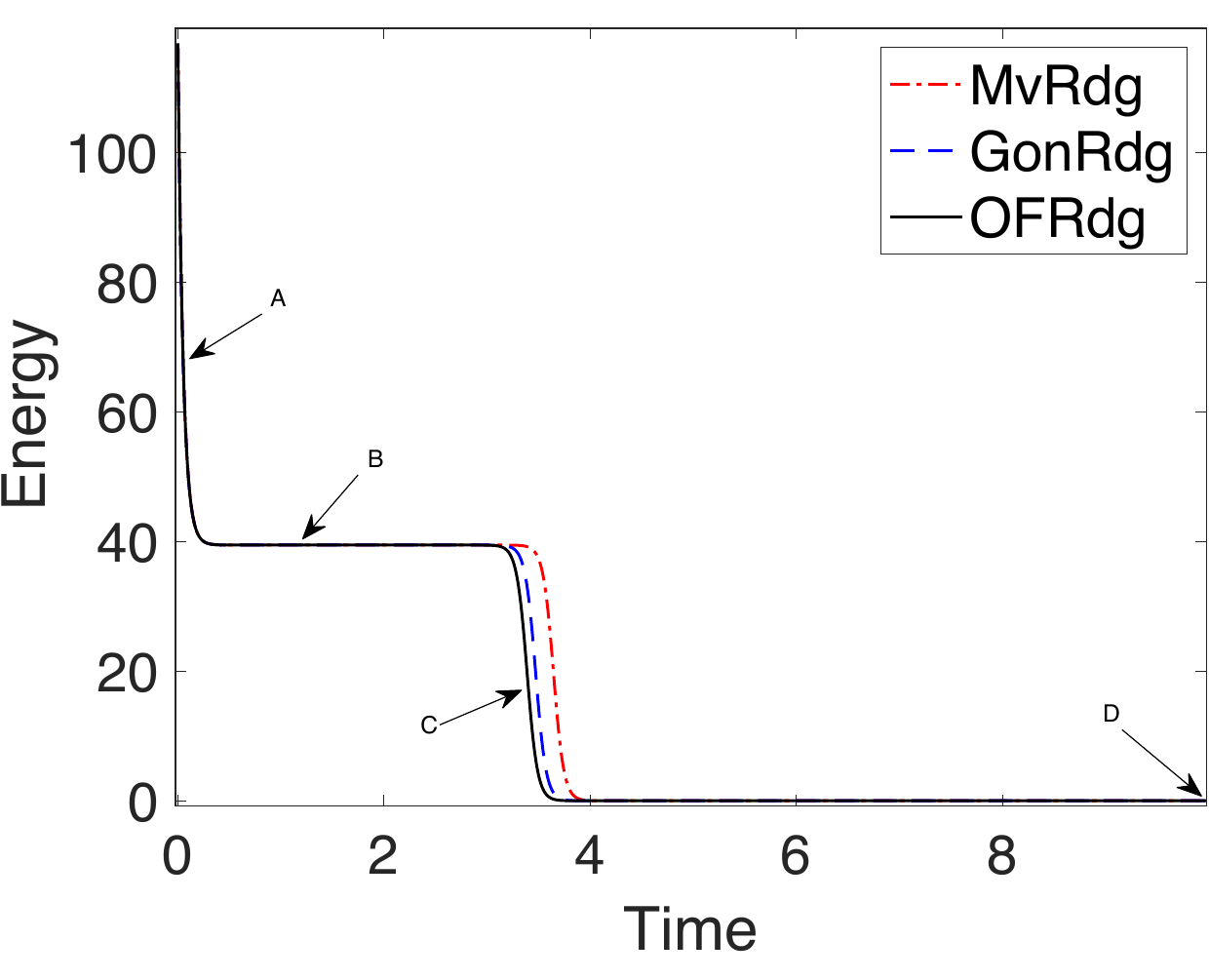}} 
 \subfigure[Length error vs Time]{  \includegraphics[scale=.28]{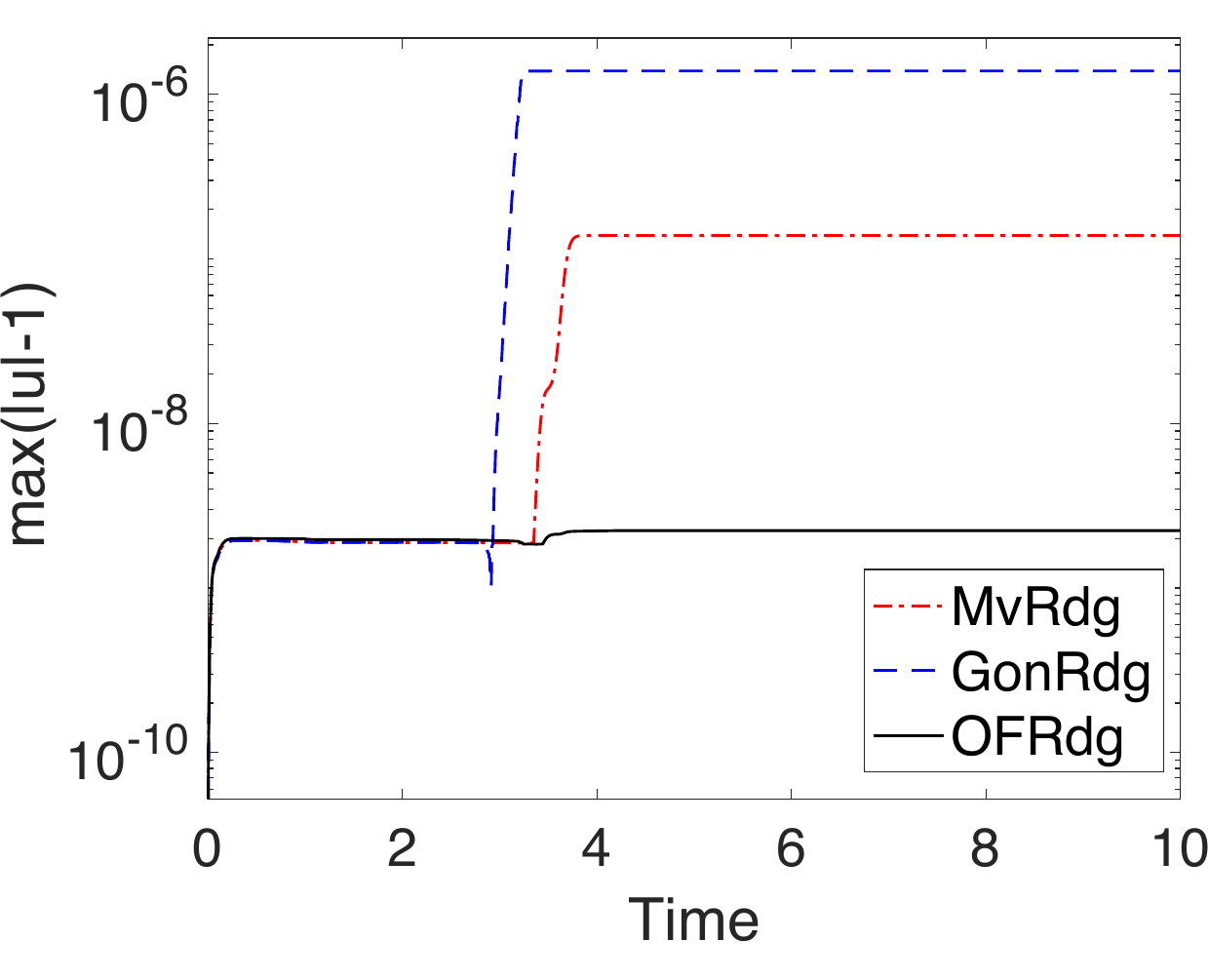} } 
 \subfigure[Computational cost vs Time]{  \includegraphics[scale=.28]{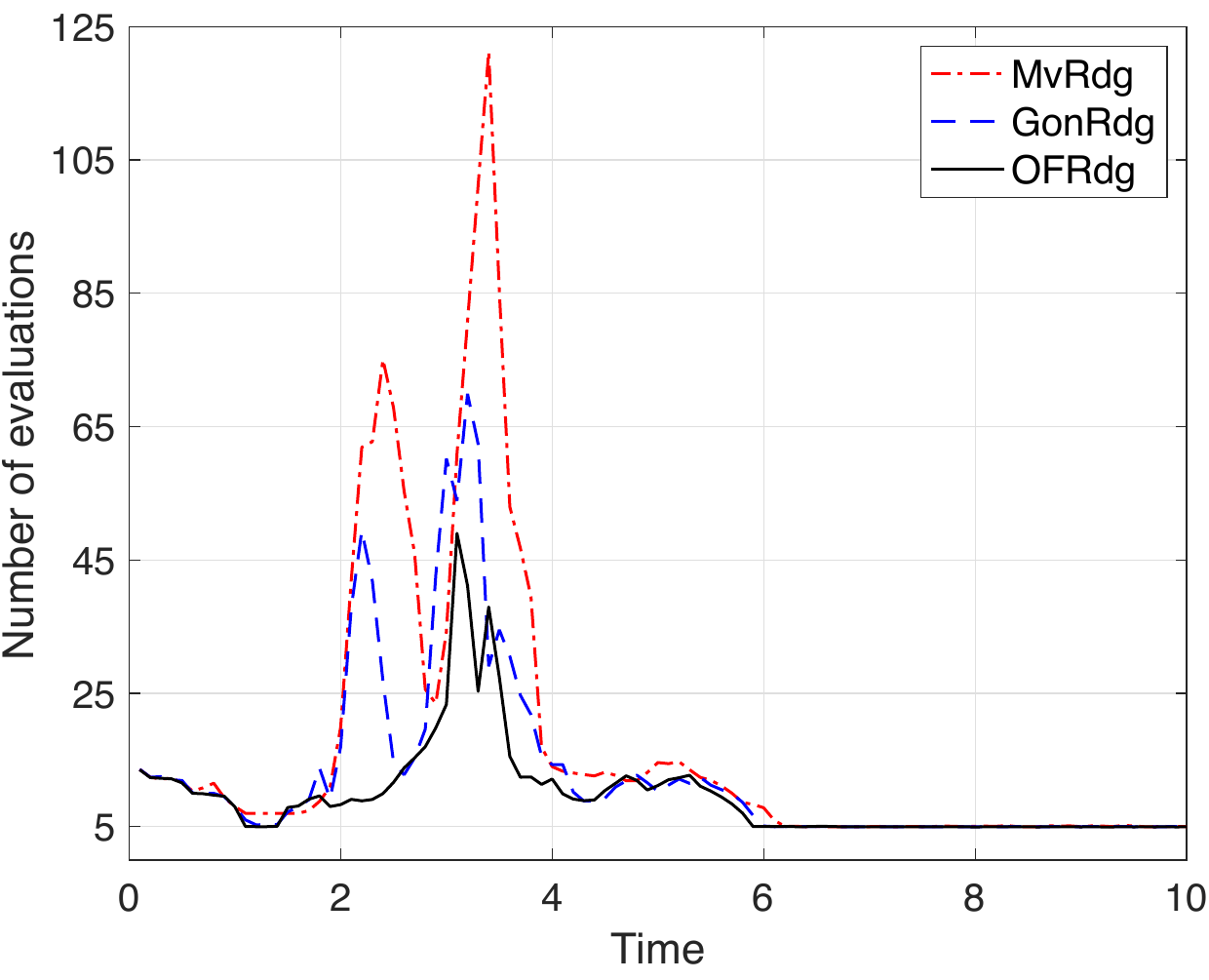} } 
 \caption{ \small Evolution of the vector field with initial field distribution \eqref{eq:utest2}. (a) Initial profile of the vector field; (b) time history of energy; (c) time history of the length error; (d) time history of the number of function evaluations at each time step.  }   
 \label{fig: homtest2}
 \end{center}
 \end{figure}
 
We depict the side view of the initial profile of the vector field, the time histories of the energy, length error in $L^{\infty}$-norm and the number of function evaluations in each time step in  Figure \ref{fig: homtest2} (a)-(d), respectively, obtained with the MvRdg, GonRdg and OFRdg methods. We observe that: 
\begin{itemize}
\item[(i)] The discrete energy curves all decay without spurious oscillations, thanks to the energy-stability nature of the proposed Rdg method. The departure time from stage B is slightly different (see Figure \ref{fig: homtest2}) due to the saddle-point nature of stage B. As we will see later, saddle-point stages are frequently encountered. 
\item[(ii)] There is a sudden jump of the length error to around $10^{-6}$ of the MvRdg and GonRdg methods, demonstrating a loss of accuracy in preserving the unit-length of the vector field, due to the convergence issue of the INK solver for these two methods. On the contrary, the length error curve obtained with the OFRdg method levels off at $10^{-9}$, which is more accurate than the previous two methods. 
\item[(iii)] Compared with the MvRdg and GonRdg methods, it takes the smallest number of function evaluations for the OFRdg method to converge in each time step.  We record the average time-consumed per time step and find that it takes 0.3194s, 0.1793s and 0.0788s for the MvRdg, GonRdg and the OFRdg methods to converge for each step on average. There is a $4\times$ speedup of the OFRdg method over the MvRdg method and a $2.2\times$ speedup over the GonRdg method.  
\end{itemize}
The above numerical experiments indicate that the OFRdg method has significant advantages over the MvRdg and GonRdg methods in terms of accuracy, efficiency and robustness.

\subsection{Dynamics against anisotropic elasticity}

 \begin{figure}[htbp]
 \begin{center}  
 \subfigure[$t=0.1$]{  \includegraphics[scale=.4]{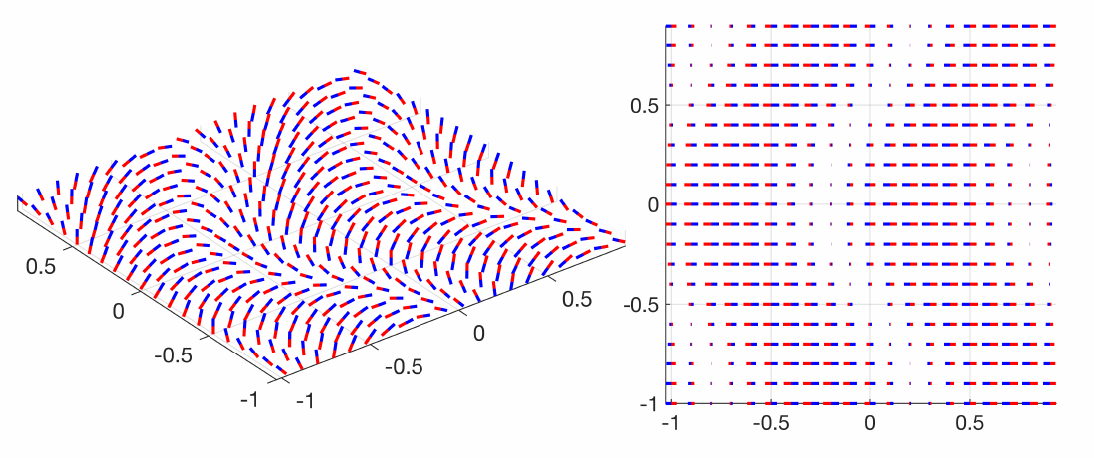}} 
 \subfigure[$t=1.5$]{  \includegraphics[scale=.4]{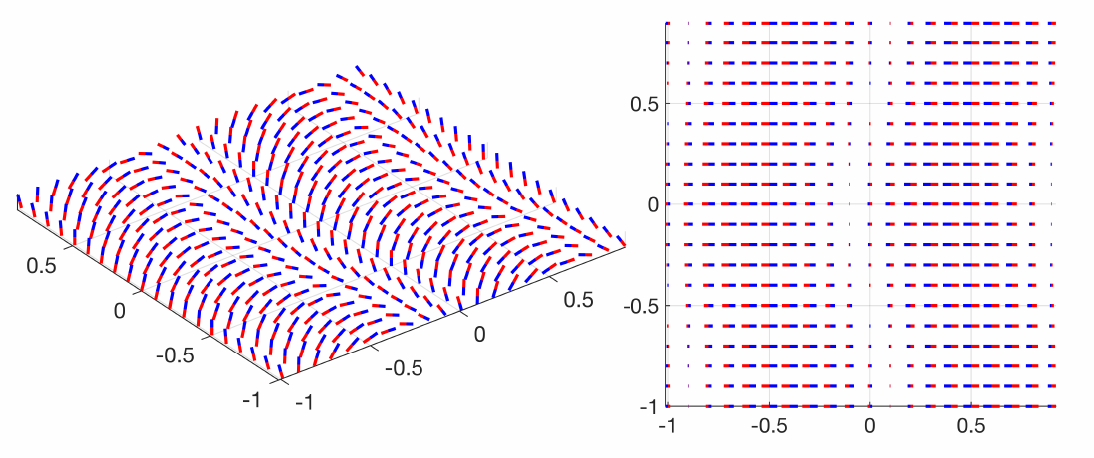} } 
  \subfigure[$t=3.5$]{  \includegraphics[scale=.4]{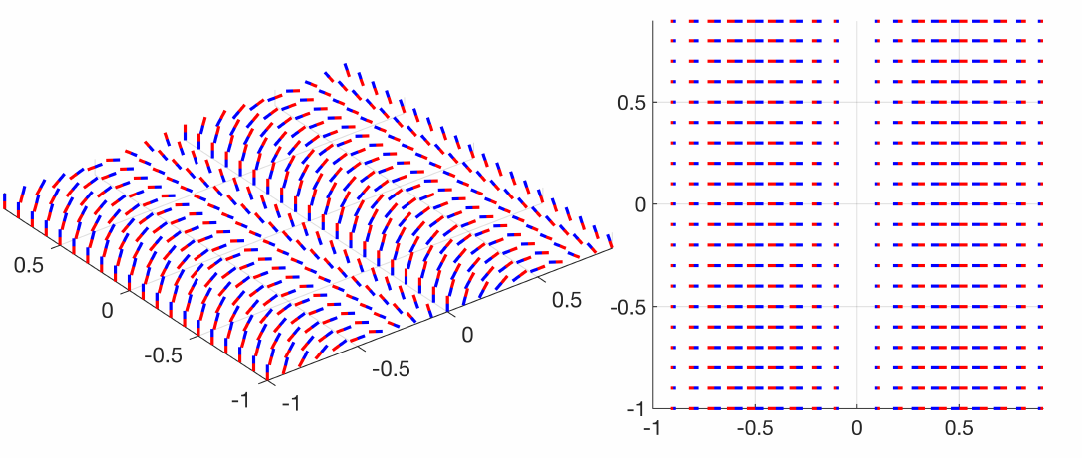} } 
   \subfigure[$t=10$]{  \includegraphics[scale=.4]{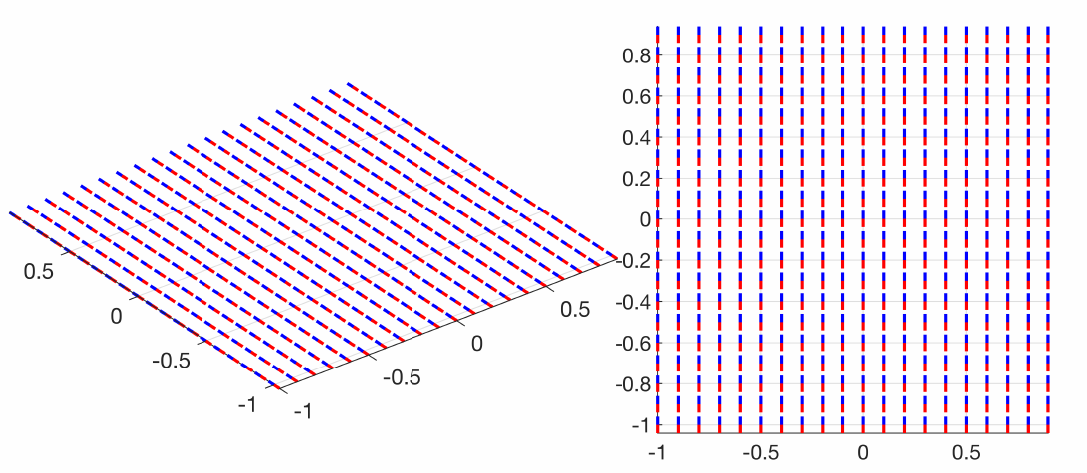} } 
 \caption{\small Time histories of snapshots of the vector field with initial field distribution \eqref{eq:utest2} and the moduli coefficients $(k_1,k_2,k_3)=(1,1,1)$ at (a) $t=0.1$; (b) $t=1.5$; (c) $t=3.5$; (d) $t=10$.}  
  \label{fig: snapshotcase1k1}
 \end{center}
 \end{figure}

\begin{figure}[htbp]
 \begin{center}  
 \subfigure[$t=10$]{  \includegraphics[scale=.4]{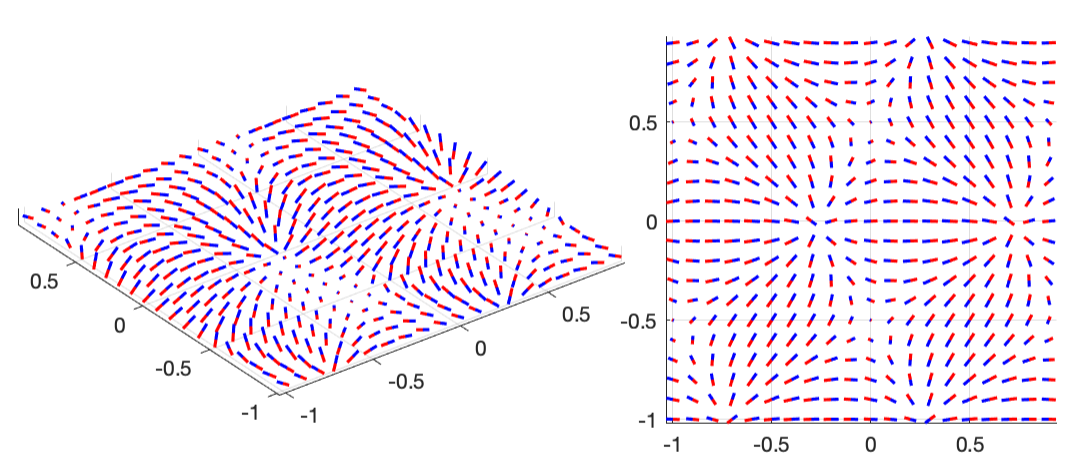}} 
   \subfigure[$t=29.7$]{  \includegraphics[scale=.4]{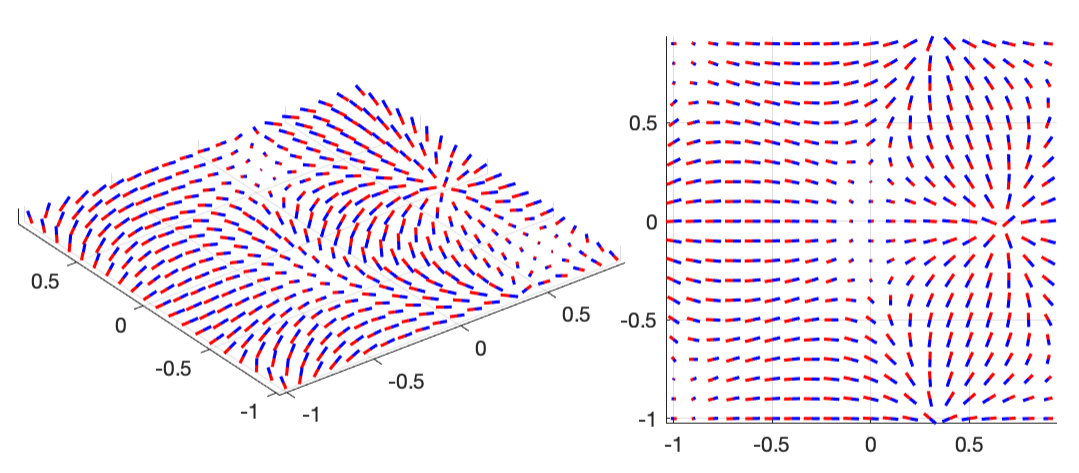}} 
  \subfigure[$t=30$]{  \includegraphics[scale=.4]{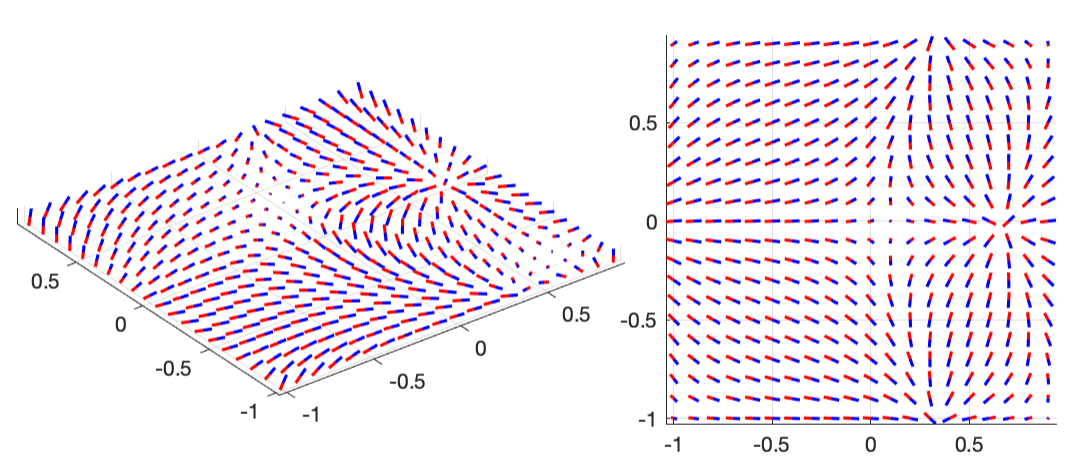}} 
    \subfigure[$t=30.5$]{  \includegraphics[scale=.4]{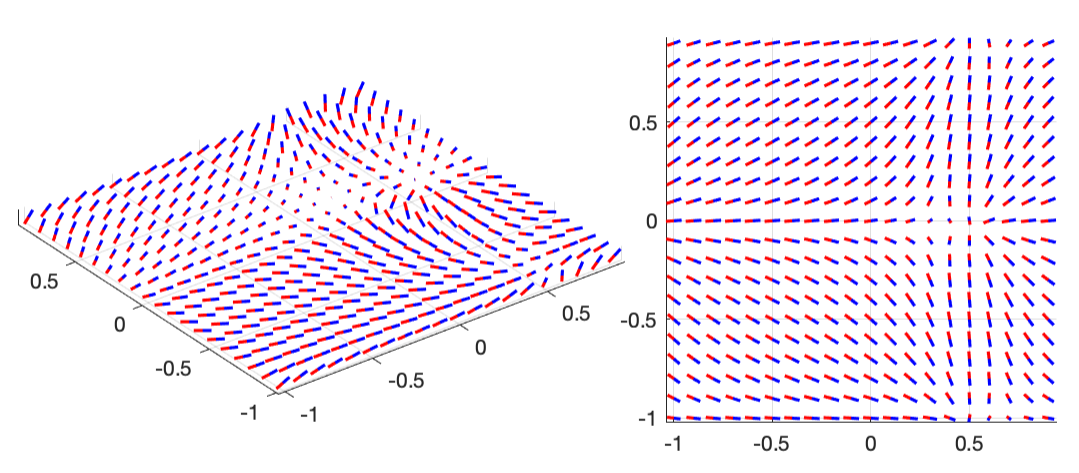}} 
        \subfigure[$t=31.5$]{  \includegraphics[scale=.4]{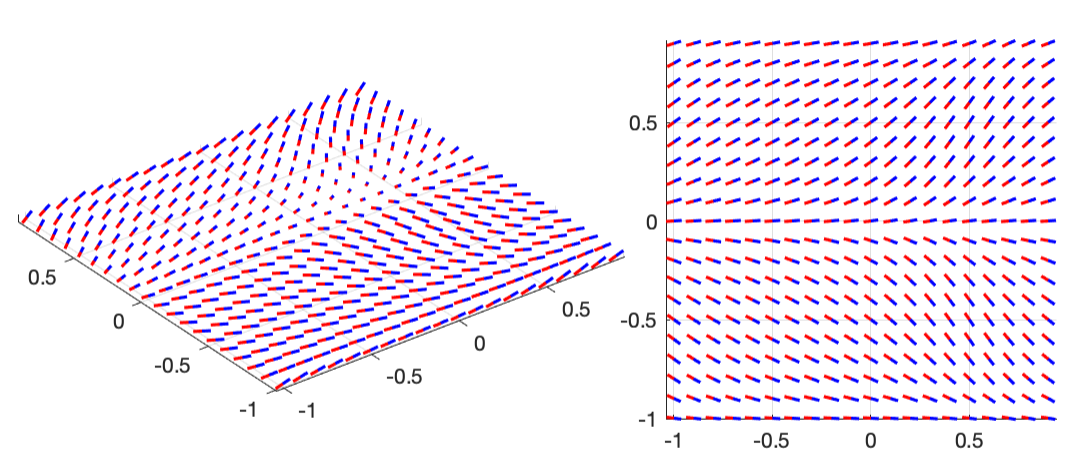}} 
    \subfigure[$t=40$]{  \includegraphics[scale=.4]{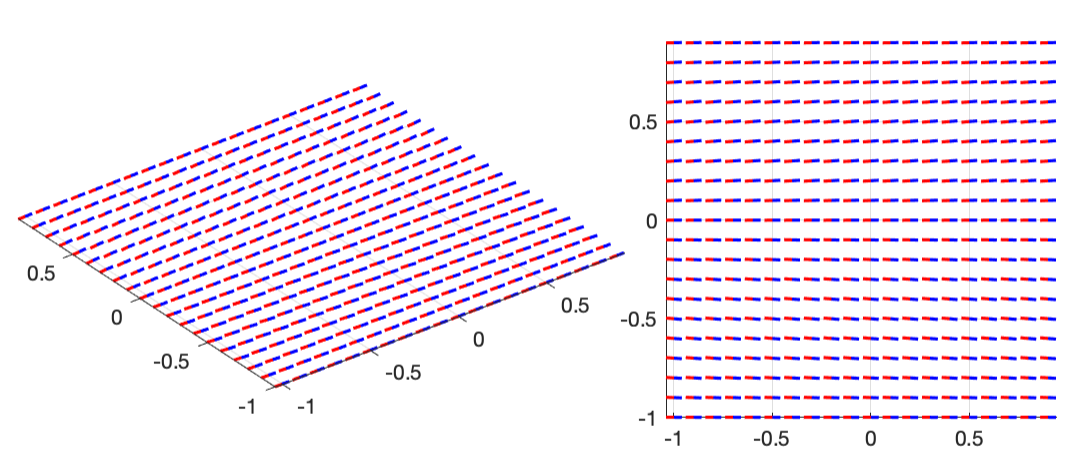}} 
 \caption{ \small Time histories of snapshots of the vector field with initial field distribution \eqref{eq:utest2} and the moduli coefficients $(k_1,k_2,k_3)=(0.08,1,1)$ at (a) $t=10$; (b) $t=29.7$; (c) $t=30$; (d) $t=30.5$; (e) $t=31.5$; (f) $t=40$. }   
 \label{fig: snapshotcase1k1008}
 \end{center}
 \end{figure}

  \begin{figure}[htbp]
 \begin{center}  
  \subfigure[Energy vs Time]{  \includegraphics[scale=.28]{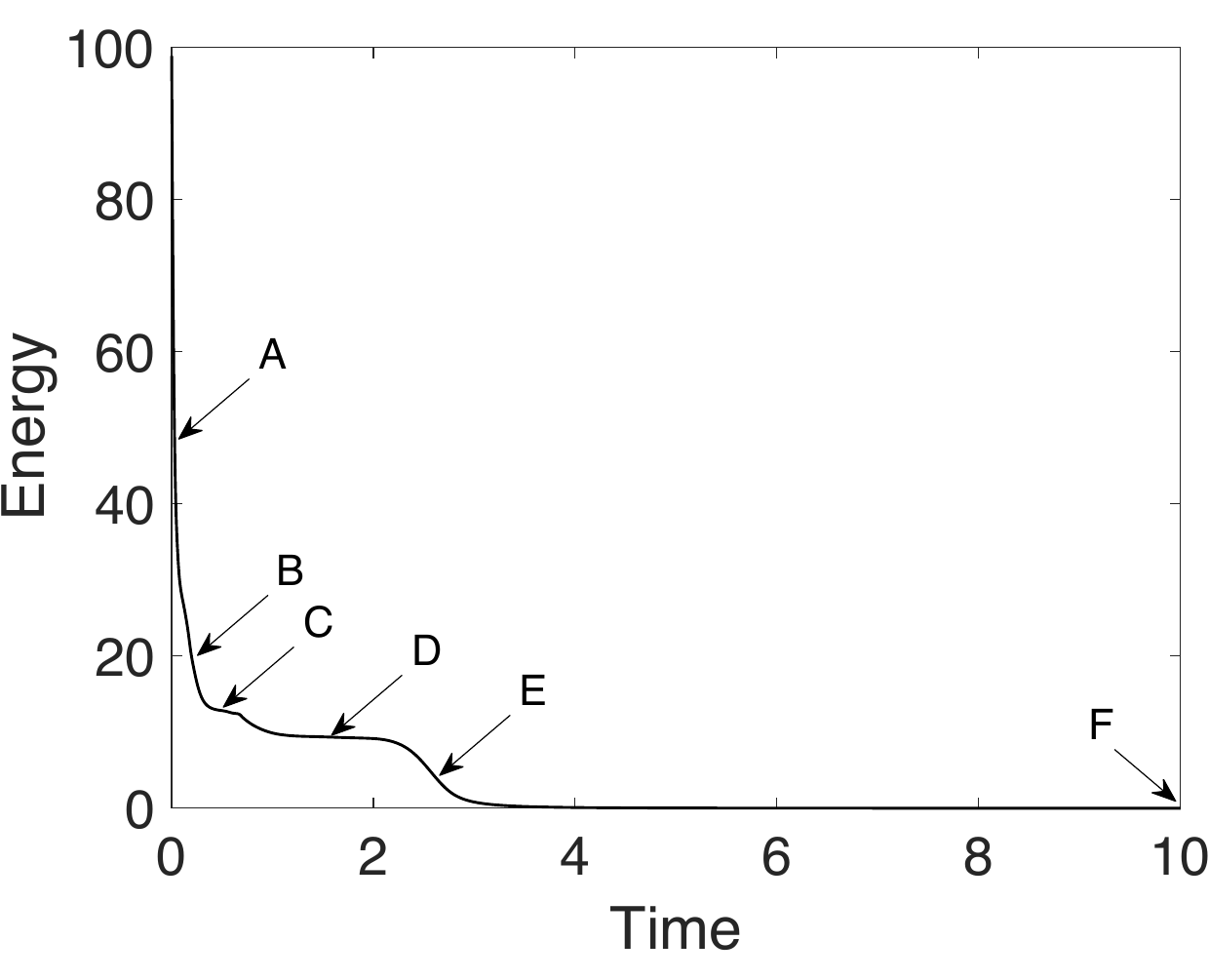}
 }   \qquad  \qquad
 \subfigure[Step size vs Time]{  \includegraphics[scale=.28]{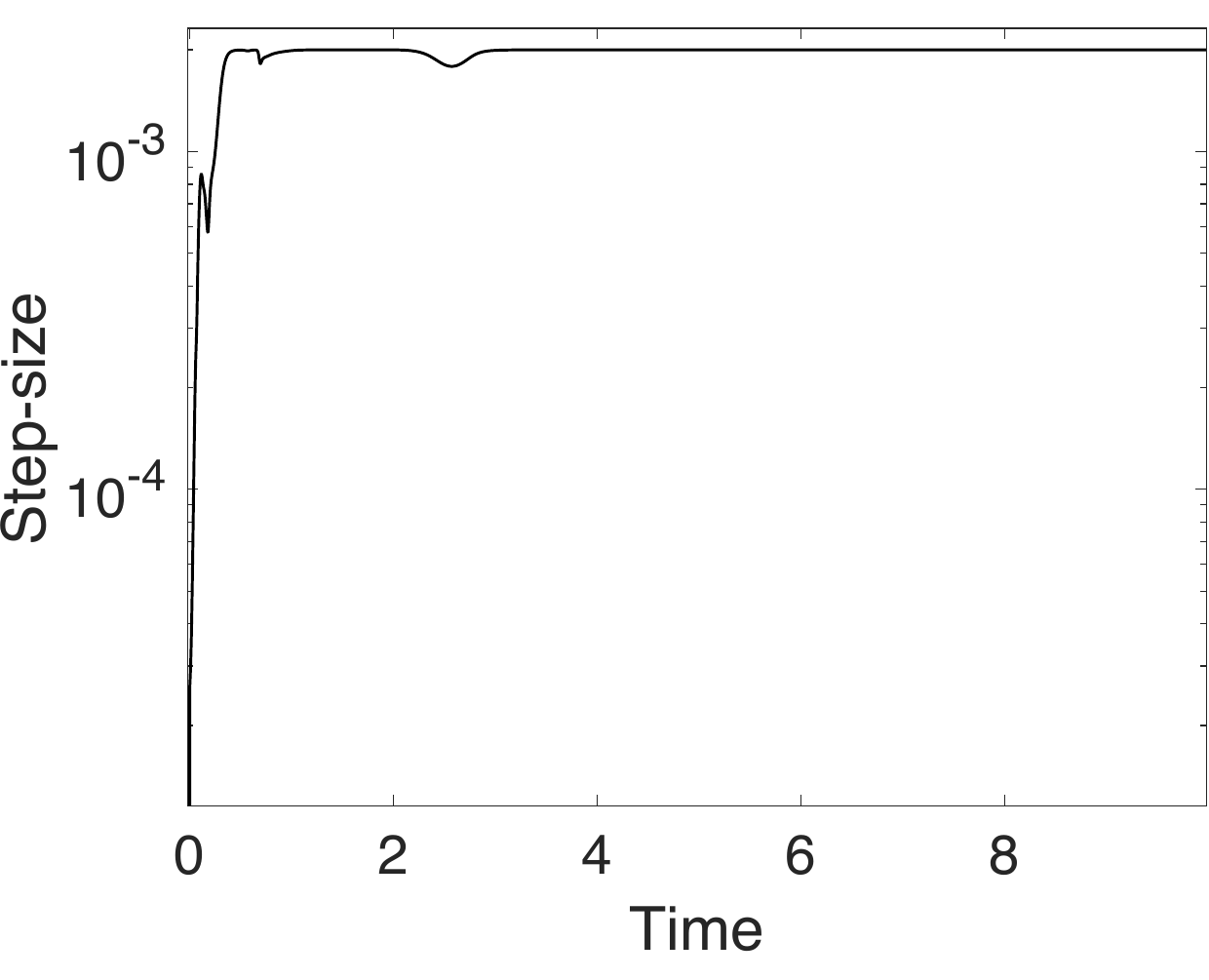}
 } 
 \subfigure[$t=0.0248$]{  \includegraphics[scale=.40]{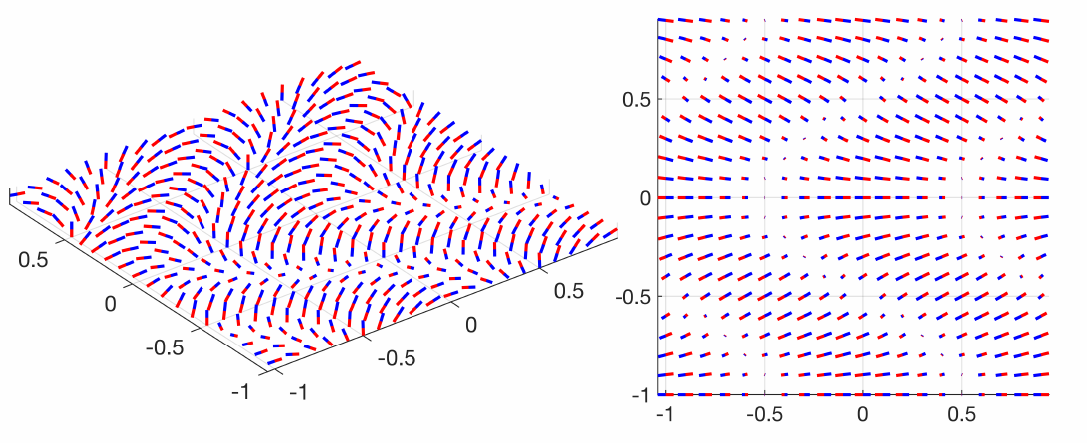}} 
 \subfigure[$t=0.1976$]{  \includegraphics[scale=.40]{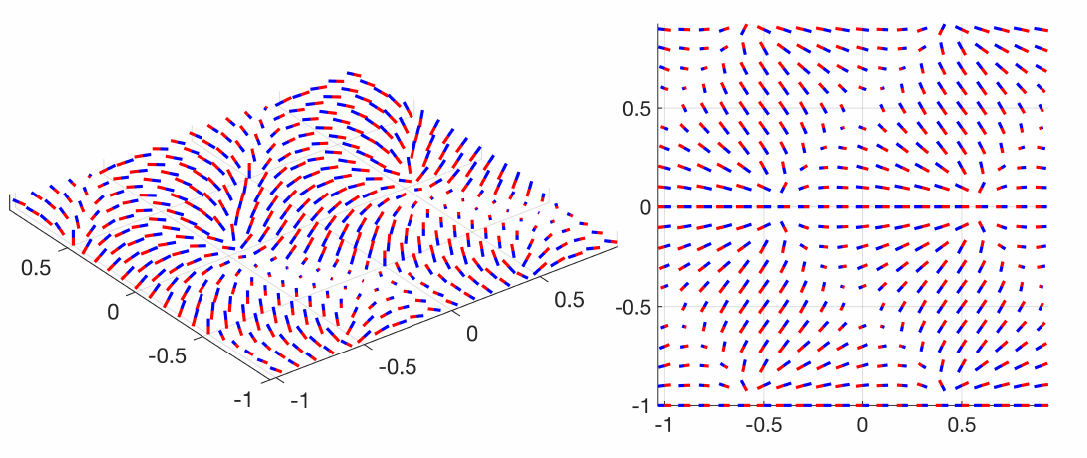} } 
  \subfigure[$t=0.4644$]{  \includegraphics[scale=.40]{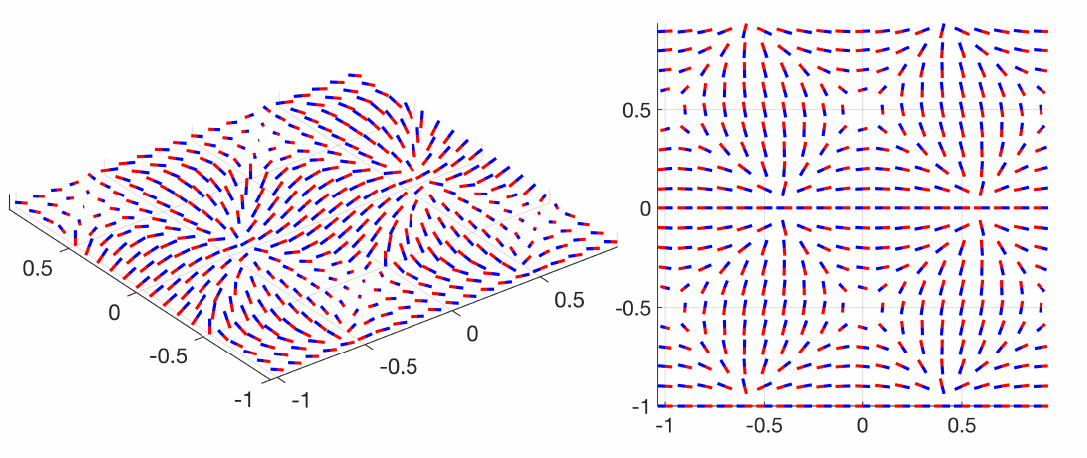} } 
   \subfigure[$t=1.9988$]{  \includegraphics[scale=.40]{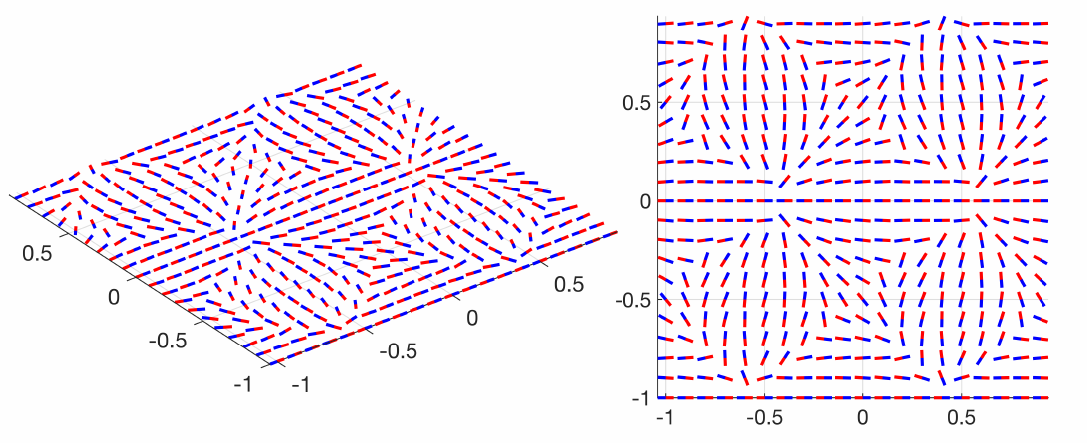} } 
      \subfigure[$t=3.2993$]{  \includegraphics[scale=.40]{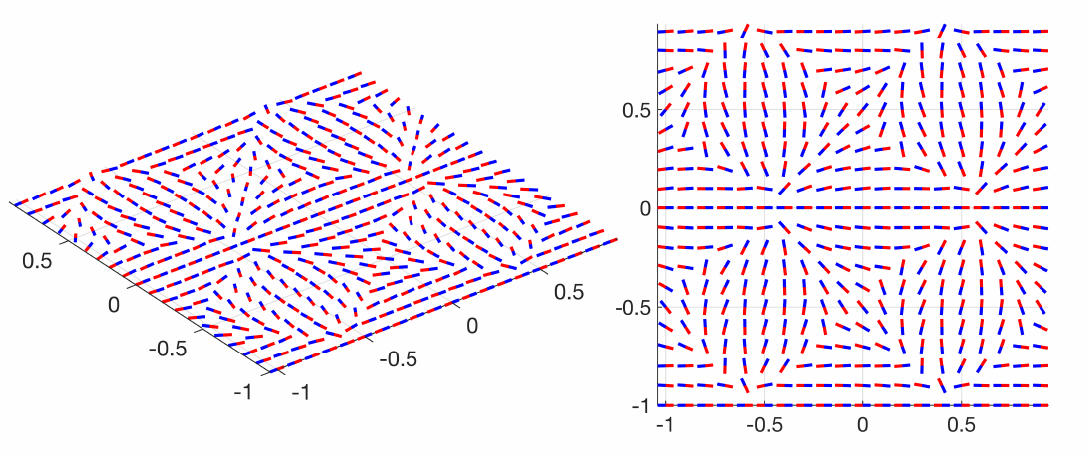} } 
         \subfigure[$t=10$]{  \includegraphics[scale=.40]{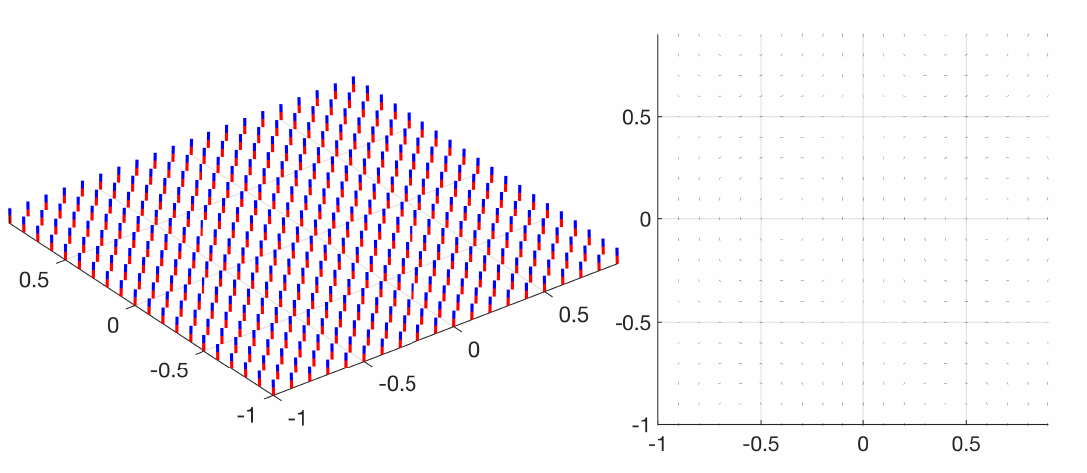} } 
 \caption{\small Evolution of the vector field with initial field distribution \eqref{eq:utest2} and the moduli coefficients $k_1=0.01, k_2=k_3=1$. (a) time history of energy; (b) time history of the step sizes obtained by the time-adaptivity strategy. Time histories of snapshots of the vector field at (c) $t=0.0248$; (d) $t=0.1976$; (e) $t=0.4644$; (f) $t=1.9988$; (g) $t=3.2993$; (h) $t=10$.   }  
  \label{case100111snapshot} 
 \end{center}
 \end{figure}


 \begin{figure}[htbp]

 \begin{center}  
 \subfigure[Energy vs Time]{  \includegraphics[scale=.28]{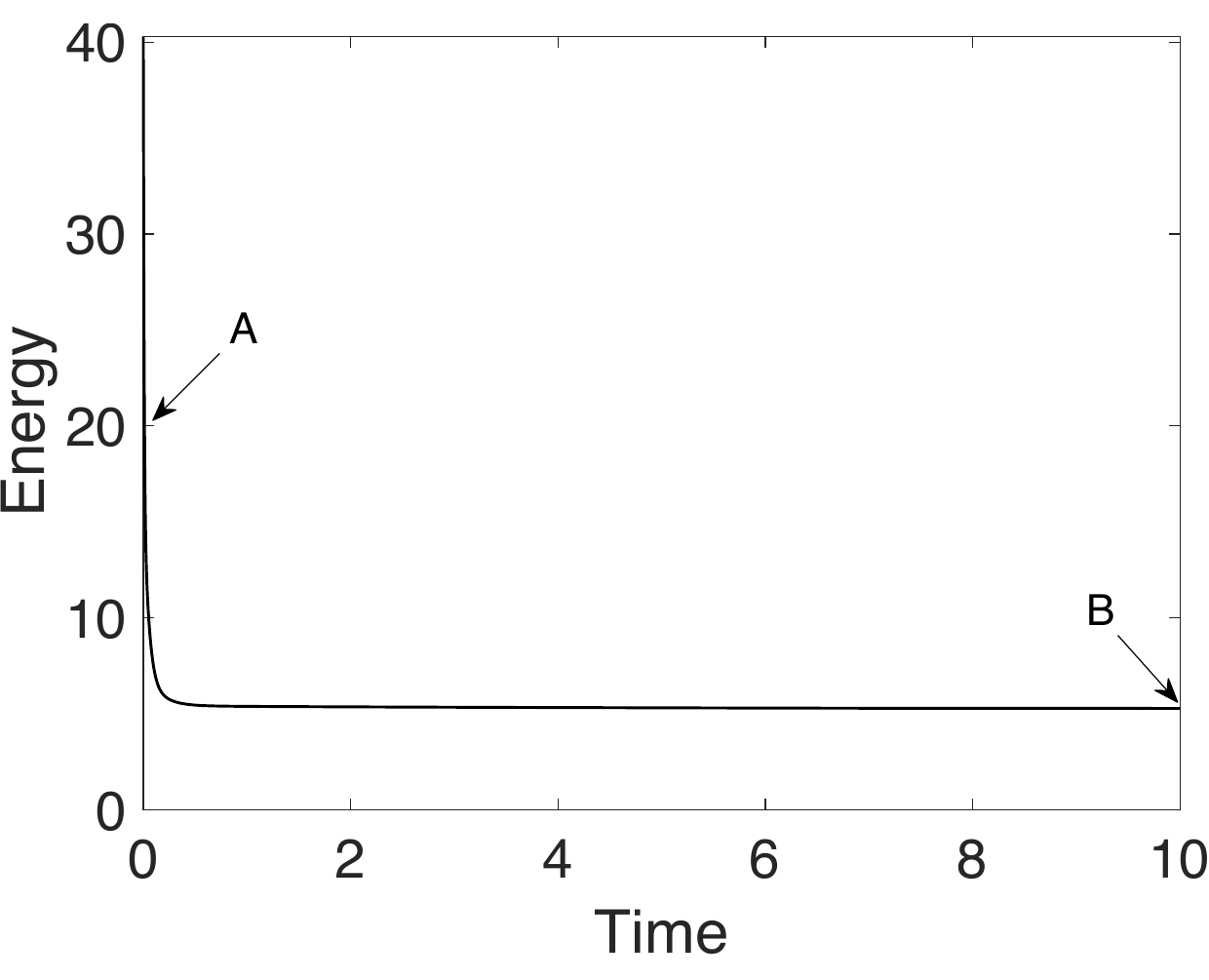} }  \qquad  \qquad
 \subfigure[Step size vs Time]{  \includegraphics[scale=.28]{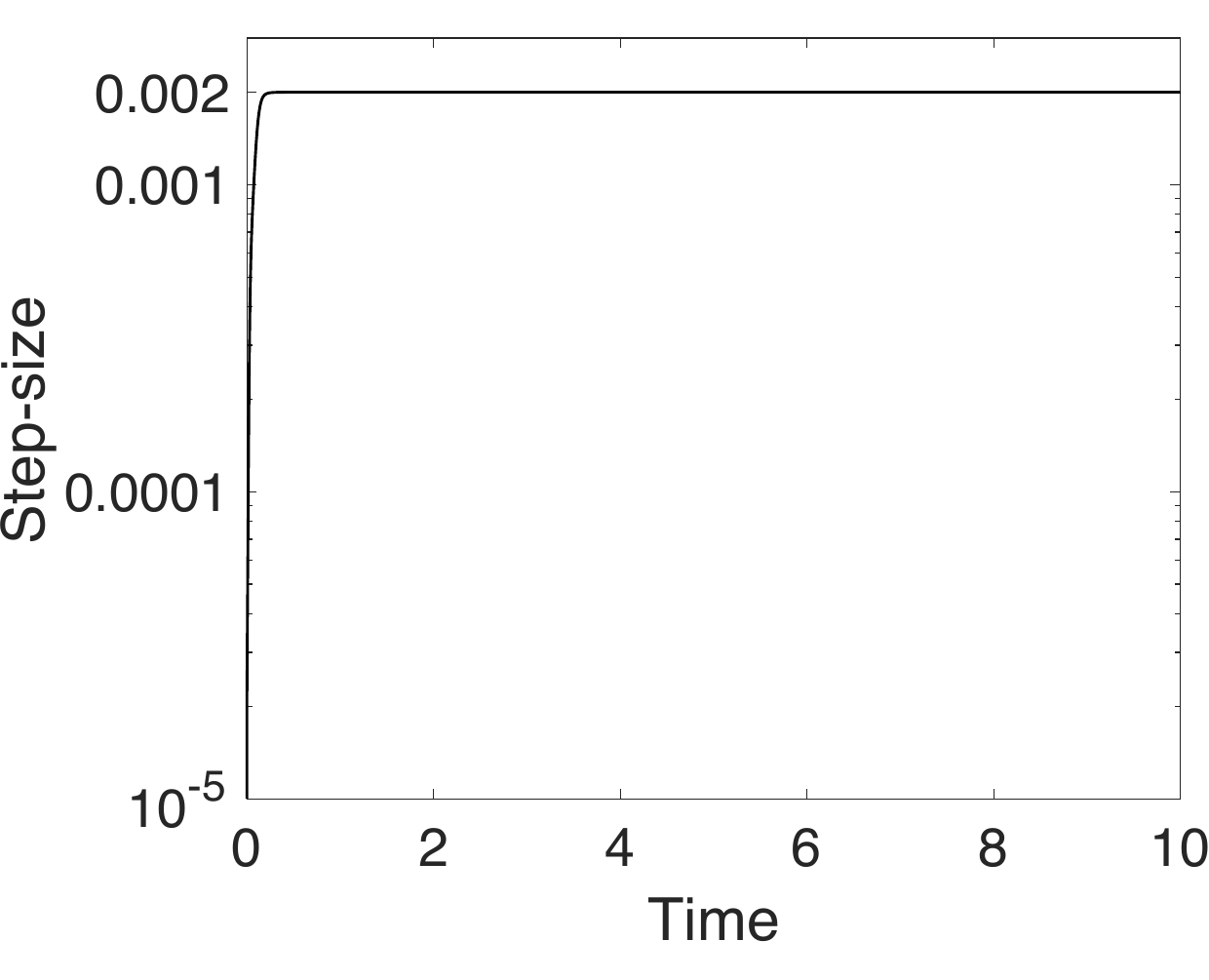} } 
  \subfigure[$t=0.0106$]{  \includegraphics[scale=.40]{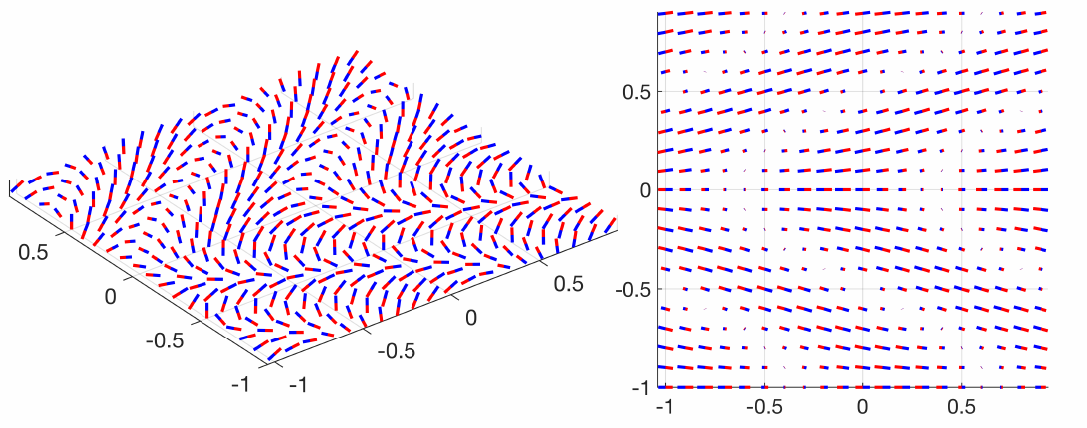}} 
 \subfigure[$t=10$]{  \includegraphics[scale=.40]{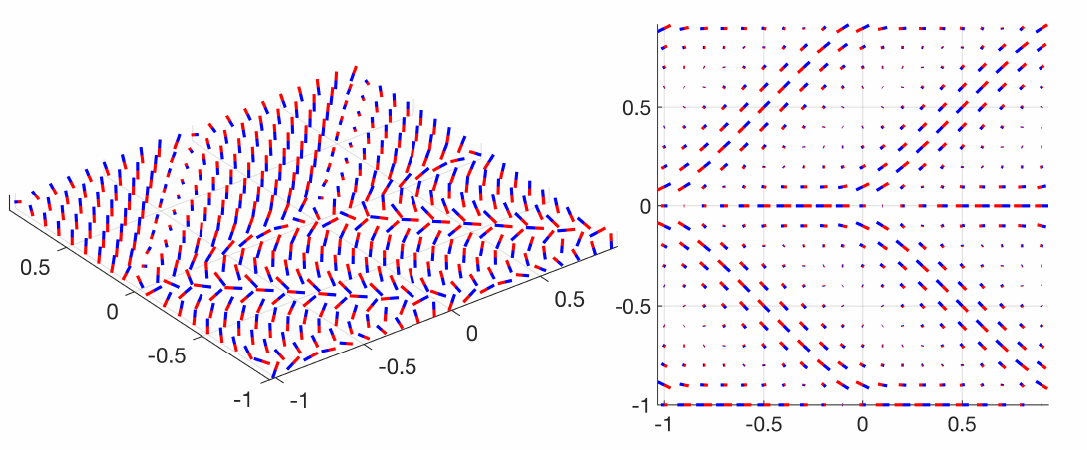} } 
 \caption{\small  Evolution of the vector field with initial field distribution \eqref{eq:utest2} and the moduli coefficients $k_1=k_3=1, k_2=0.01$. (a) time history of energy; (b) time history of the step sizes obtained by the time-adaptivity strategy;  Time histories of snapshots of the vector field at (c) $t=0.0106$ and (d) $t=10$.  }  
  \label{case110011snapshot}
 \end{center}
 \end{figure}

 \begin{figure}[htbp]
 \begin{center}  
 \subfigure[Energy vs Time]{  \includegraphics[scale=.28]{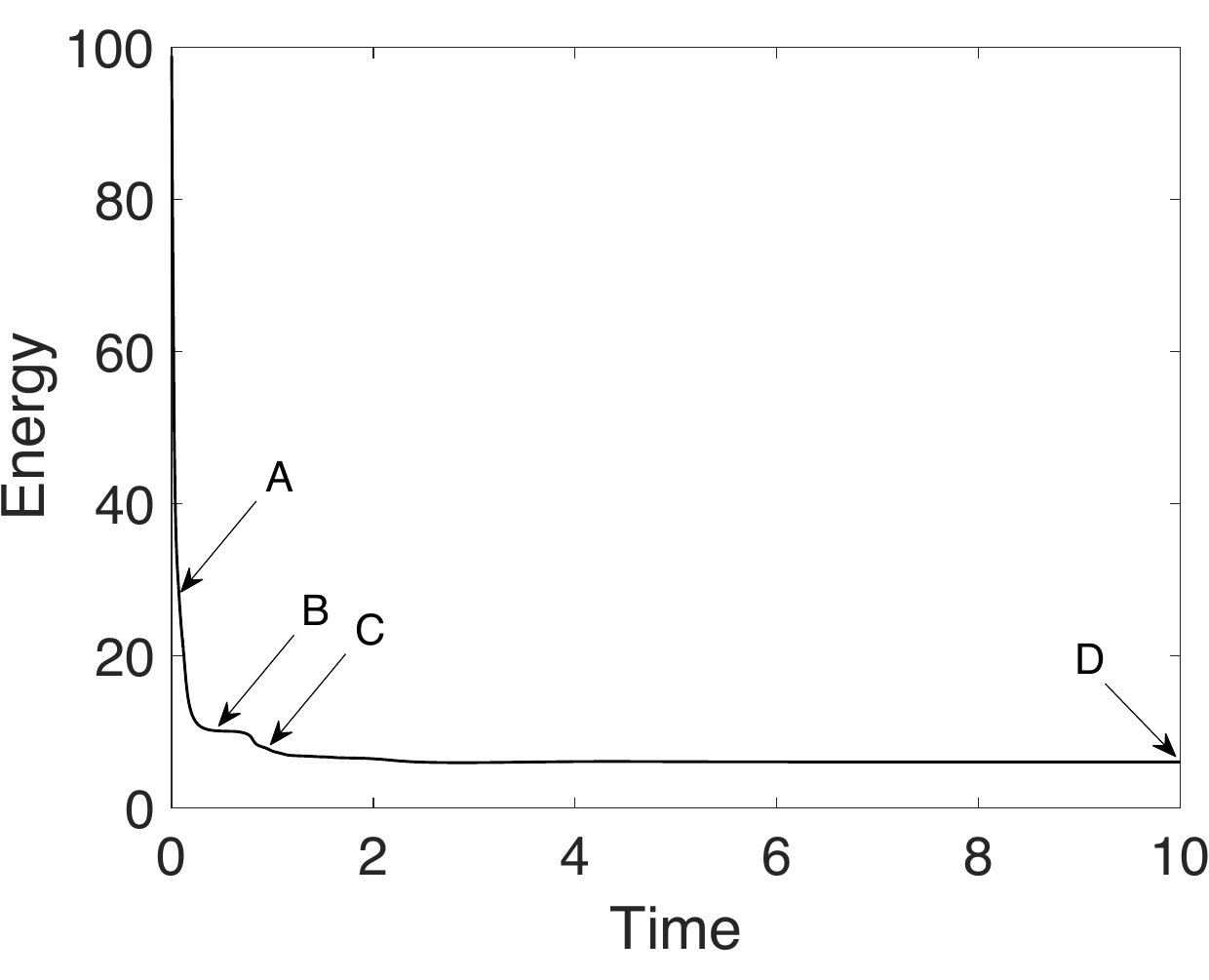} } \qquad  \qquad
 \subfigure[Step size vs Time]{  \includegraphics[scale=.28]{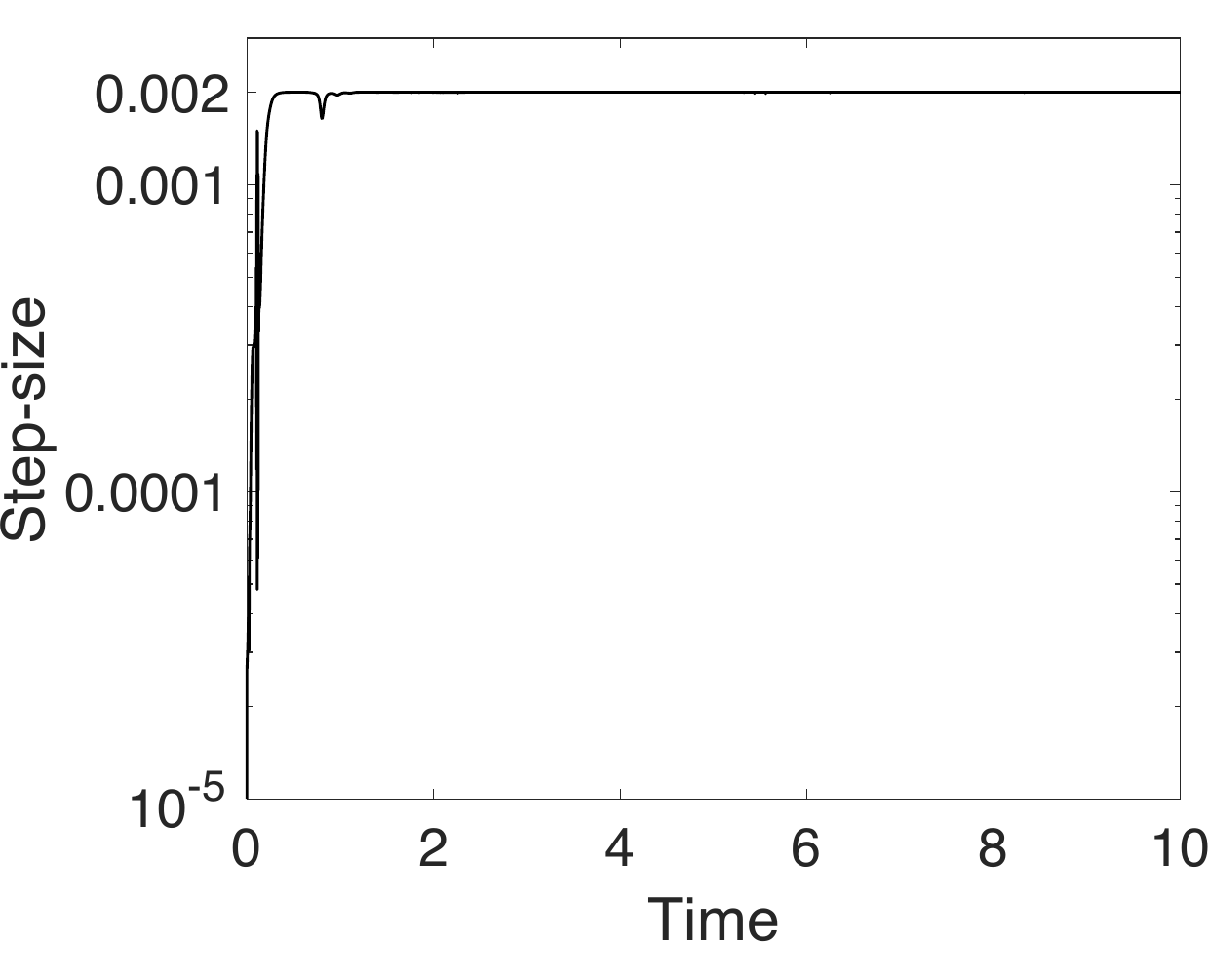} } 
  \subfigure[$t=0.0639$]{  \includegraphics[scale=.40]{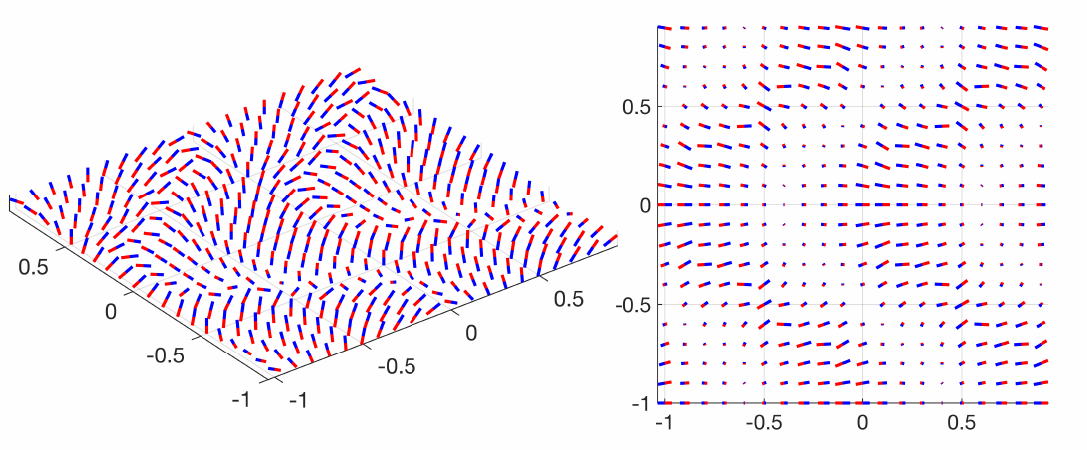}} 
 \subfigure[$t=0.6776$]{  \includegraphics[scale=.40]{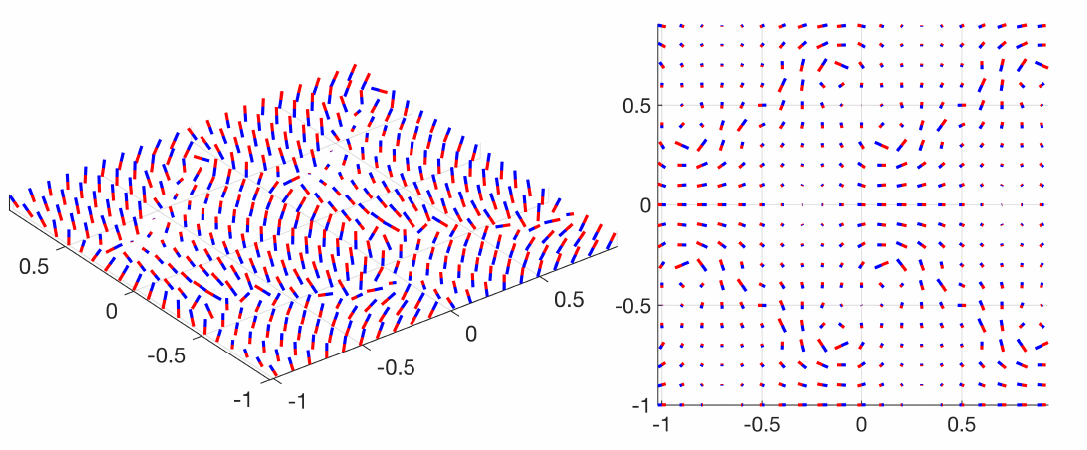} } 
   \subfigure[$t=1.0004$]{  \includegraphics[scale=.40]{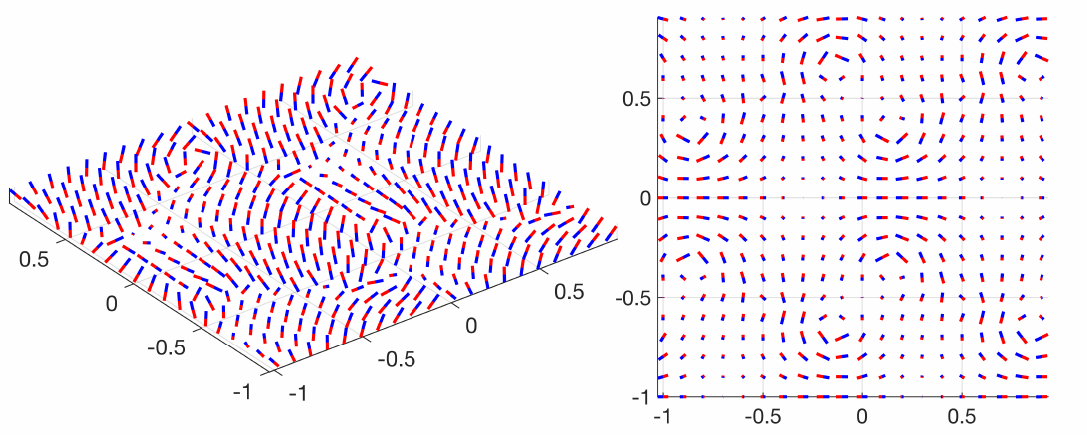}} 
 \subfigure[$t=10$]{  \includegraphics[scale=.40]{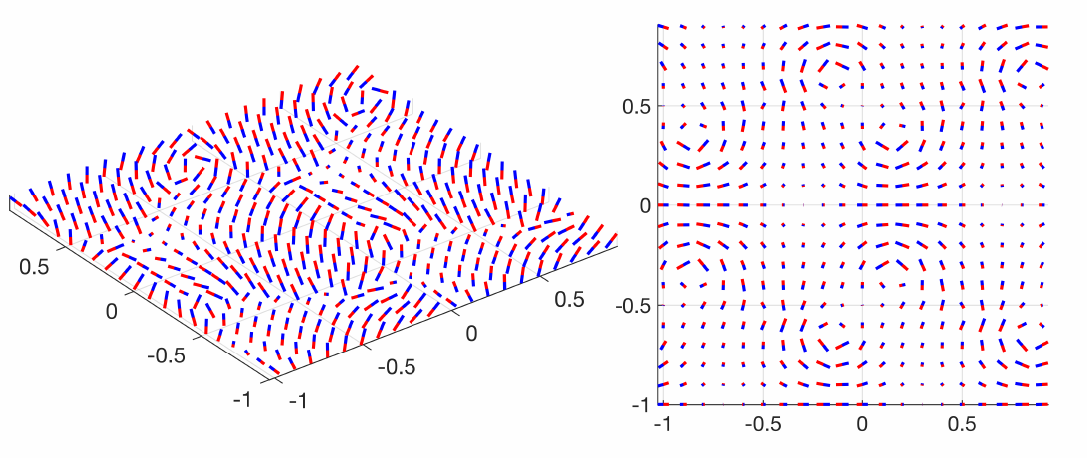} } 
 \caption{\small  Evolution of the vector field with initial field distribution \eqref{eq:utest2} and the moduli coefficients $k_1=k_2=1, k_3=0.01$. (a) time history of energy; (b) time history of the step sizes obtained by the time-adaptivity strategy;  Time histories of snapshots of the vector field at (c) $t=0.0639$; (d) $t=0.6776$; (e) $t=1.0004$ and (f) $t=10$.  }  
  \label{case111001snapshot}
 \end{center}
 \end{figure}

In the following computation, we adopt the OFRdg method with the proposed time-adaptive strategy in Section \ref{sect: adap} to investigate transitions in the dynamics of vector field under various elastic coefficients $k_i$. The initial distribution of the vector field is prescribed as in equation \eqref{eq:utest2}, the time-adaptivity parameters are set to $\tau_{\rm max}=2\times10^{-3}$, $\tau_{\rm min}=10^{-5}$ and $\alpha=10^{-3}$ in equation \eqref{eq: adap} and the simulation settings are kept the same with Section \ref{sect: ppt}.

Figure \ref{fig: snapshotcase1k1} depicts the evolution process of the nematic molecules by a temporal sequence of snapshots of the vector field with $(k_1,k_2,k_3)=(1,1,1)$. Both the side and top views of the vector field in alignment  are provided at time $t=0.1$, $t=1.5$, $t=3.5$ and $t=10$, corresponding to the fast transient stage A, the saddle-point stage B, the transient stage C and the stable equilibrium stage D, as are marked in Figure \eqref{fig: homtest2} (b). It can be observed that the vector field first becomes homogeneous in the $x_2$-direction, with the directions rotated in the $x_1$-$x_3$ plane.
During this period, the $x_2$-component of the vector field maintains zero as it initially stands. 
Then, the $x_2$-component rapidly grows and dominates.
The vector field eventually becomes homogeneous along the $x_2$-direction.  

We decrease the value of $k_1$ and find that the same dynamical behaviors with $k_1=0.5,0.2,0.1$.
However, when $k_1$ decreases to $0.08$, a different dynamics is observed. 
The $x_2$-component emerges at the beginning, reaching the state shown in Figure \ref{fig: snapshotcase1k1008} (a) at which the system remains quite long. 
If we look into the vector field with the top view from the $z$-direction, i.e. the projection of the vector field on the $x_1$-$x_2$ plane, we find a few nodal patterns.
Then, these nodal patterns successively disappear in a relatively short period of time, as can be seen in Figure \ref{fig: snapshotcase1k1008} (b)-(e), before finally arriving at a homogeneous vector field (Figure \ref{fig: snapshotcase1k1008} (f)).
When we further decrease $k_1$ to $0.06$, the state in Figure \ref{fig: snapshotcase1k1008} (a) remain until $t=50$ and perturbations cannot drive the system away from this state, which implies that this state becomes a local energy minimizer.
A smaller value $k_1=0.01$ leads to a more complex dynamical process. Figure \ref{case100111snapshot} (a) depicts the time history of energy as a function of time. We find that the evolution process can be divided into six stages: transient stage A (corresponds to Figure \ref{case100111snapshot} (c));  slightly slower transient stage B (Figure \ref{case100111snapshot} (d)); saddle-point stage C  (Figure \ref{case100111snapshot} (e))  resembling Figure \ref{fig: snapshotcase1k1008} (a); another saddle-point stage D  (Figure \ref{case100111snapshot} (f)) where the vector field mostly lying in the $x_1$-$x_2$ plane; a slow transition stage E (Figure \ref{case100111snapshot} (g)) and global equilibrium stage F (Figure \ref{case100111snapshot} (h)).    

We also investigate the cases where $k_2$ or $k_3$ becomes small. We fix $(k_1,k_2,k_3)=(1,0.01,1)$ and $(k_1,k_2,k_3)=(1,1,0.01).$ In Figure \ref{case110011snapshot} and Figure \ref{case111001snapshot}, respectively and depict the evolution of energy and time histories of snapshots for the vector field and completely different evolution phenomenon can be observed.  Small perturbations on the orientation of the director fields cannot drive the system away from state B (see Figure \ref{case110011snapshot} (d)) and state D (see Figure \ref{case111001snapshot} (d)) for these two cases, respectively, which implies that these states become local energy minimizers. The above examples clearly indicate that the dynamics driven by anisotropic elasticity could be totally distinct, which cannot be covered here and calls for further studies.

Moreover, we plot in Figure \ref{case100111snapshot} (b), Figure \ref{case110011snapshot} (b) and Figure \ref{case111001snapshot} (b) the time step sizes as a function of time for these three cases. It can be observed that the step sizes change adaptively according to the dynamics of the vector field and the step sizes remain $\tau_{\rm max}=2\times 10^{-3}$ for most of the time steps. This clearly demonstrates that the efficiency and accuracy of the proposed OFRdg method with time-adaptive strategy.

\section{Concluding Remarks}\label{sect: concl}
We propose a second-order rotational discrete gradient method for the Oseen--Frank gradient flow, unconditionally preserving the unit vector constraint and energy dissipation.
The scheme is built by rewriting the gradient flow in the rotational form.
A discrete gradient specifially designed for the Oseen--Frank energy is given, which proves to be more efficient and robust than the two general formulae for discrete gradients by several numerical experiments. 
The role of anisotropic elasticity is examined in a couple of cases, where distinct dynamical behaviors are observed.
The effect of anisotropic elasticity on dynamics would be an interesting problem to investigate further, for which the scheme we propose here is believed to be a useful tool.
The scheme is also expected to be a fundamental step for the design of efficient and robust schemes for the Ericksen--Leslie equation involving the incompressible flow, which we aim to study in the near future. 

\section{Appendix: non-periodic boundary conditions}
Let $\bs{n}$ pertubate along $\delta\bs{n}$. The derivative is calculated as 
\begin{align}
  &\lim_{s\to 0} \frac{\mathcal{F}[\bs{n}+s\delta\bs{n}]-\mathcal{F}[\bs{n}]}{s}\nonumber\\
  =\,&\int_{\Omega}k_1(\nabla\cdot\bs{n})(\nabla\cdot\delta\bs{n})+k_2(\bs{n}\cdot \nabla \times \bs{n})(\delta\bs{n}\cdot\nabla \times \bs{n}+\bs{n}\cdot\nabla \times \delta\bs{n})   \nonumber\\
  &\quad +k_3(\bs{n}\times \nabla \times \bs{n})\cdot (\delta\bs{n}\times\nabla \times \bs{n}+\bs{n}\times\nabla\times\delta\bs{n})dV\nonumber\\
  =\,&\int_{\Omega}\delta\bs{n}\cdot\bigg[-k_{1}\nabla(\nabla \cdot \bs{n})+k_{2}\Big(   (\bs{n}\cdot \nabla \times \bs{n})(\nabla \times \bs{n})+   \nabla\times\big((\bs{n}\cdot \nabla \times \bs{ n})\bs{n}  \big)        \Big)\nonumber\\
    &\quad +k_{3}\Big(  (\nabla \times \bs{n})\times (\bs{n}\times \nabla \times \bs{n})+   \nabla\times\big((\bs{n}\times \nabla \times \bs{ n})\times\bs{n}  \big) \Big)\bigg]dV\nonumber\\
  &+\int_{\partial\Omega}k_1(\nabla\cdot\bm{n})\bs{\nu}\cdot\delta\bs{n}
    +k_2(\bs{n}\cdot \nabla \times \bs{n})(\bs{n}\cdot\bs{\nu}\times\delta\bs{n})\nonumber\\
    &\quad +k_3(\bs{n}\times \nabla \times \bs{n})\cdot\big(\bs{n}\times (\bs{\nu} \times \delta\bs{n})\big)dS,
  \label{OSvargen}
\end{align}
where we recall that $\nu$ is the outward normal unit vector. 
Using \eqref{eq:aba} and \eqref{eq: eleid}, the surface integral can be rewritten as 
\begin{align}
  \int_{\partial\Omega}\delta\bs{n}\cdot\Big[k_1(\nabla\cdot\bm{n})\bs{\nu}
    +k_2(\bs{n}\cdot \nabla \times \bs{n})(\bs{n}\times\bs{\nu})
    +k_3\big((\bs{n}\times \nabla \times \bs{n})\times \bs{n}\big)\times\bs{\nu}\Big]dS.\label{OSvarSurf}
\end{align}
Therefore, when Dirichlet boundary conditions are adopted, on $\partial\Omega$ we have $\delta\bs{n}=0$, so that the surface integral vanishes. 
The natural boundary conditions make the surface integral vanish without requiring $\delta\bs{n}=0$. 
Since $\delta\bs{n}$ is perpendicular to $\bs{n}$, the nautral boundary conditions shall be given by 
\begin{align}
  \bs{n}\times \Big[k_1(\nabla\cdot\bm{n})\bs{\nu}
    +k_2(\bs{n}\cdot \nabla \times \bs{n})(\bs{n}\times\bs{\nu})
    +k_3\big((\bs{n}\times \nabla \times \bs{n})\times \bs{n}\big)\times\bs{\nu}\Big]=0.\label{ntrbnd}
\end{align}
Thus, for both Dirichlet boundary conditions and \eqref{ntrbnd}, the energy variation is given by \eqref{variation}.
The derivation of discrete gradient can be done similarly.
In particular, \eqref{ntrbnd} shall be discretized as 
\begin{align}
  \bs{n}^{n+1/2}\times \Big[k_1(\nabla\cdot\bm{n}^{n+1/2})\bs{\nu}
    +k_2\beta^{n+1/2}(\bs{n}^{n+1/2}\times\bs{\nu})
    +k_3\big(\bs{\omega}^{n+1/2}\times \bs{n}^{n+1/2}\big)\times\bs{\nu}\Big]=0,\label{ntrbnd_dis}
\end{align}
in order that the surface integral vanishes. 

\section*{Acknowledgement}
J. Xu acknowledges the support from the NSFC (Nos. 12288201, 12001524). Z. Yang acknowledges the support from the NSFC (No. 12101399), the Shanghai Sailing Program (No. 21YF1421000) and the Fundamental Research Funds for the Central Universities. 
 
 \bibliographystyle{plain}
\bibliography{refpapers}

\end{document}